\documentclass[11pt]{article}
\usepackage[T2A]{fontenc}
\usepackage[utf8]{inputenc}

\usepackage[main=english,russian]{babel} 

\usepackage{dsfont}
\usepackage[all]{xy}
\usepackage{microtype}

\usepackage[colorinlistoftodos]{todonotes}
\setuptodonotes{color=blue!20, textcolor=black, linecolor=blue!60!black}

\usepackage[utf8]{inputenc}
\usepackage{amsmath, amssymb, amsthm}
\usepackage{graphicx}
\graphicspath{{images/}}
\usepackage{geometry}
\usepackage{enumitem}
\usepackage{nccmath, mathtools}
\usepackage{indentfirst}
\usepackage{fancyhdr}
\usepackage{url}
\usepackage{multicol}
\usepackage{float}
\usepackage[nottoc]{tocbibind}
\usepackage[colorlinks,
            citecolor=blue,
            linkcolor=red,
            urlcolor=blue,
            pdfpagelabels,
            plainpages=false,
            hypertexnames=false]{hyperref}
\usepackage[capitalize]{cleveref}
\crefname{enumi}{}{}

\usepackage[labelfont=sc]{caption}

\newcommand{\thesistitle}[0]{On the integrality of some P-recursive sequences}
\theoremstyle{definition}
\newtheorem{defi}{Definition}[section]
\newtheorem*{ex}{Example}
\theoremstyle{plain}
\newtheorem{thm}[defi]{Theorem}

\newtheorem{lem}[defi]{Lemma}
\newtheorem{cor}[defi]{Corollary}
\newtheorem{prop}[defi]{Proposition}
\newtheorem{conj}[defi]{Conjecture}
\theoremstyle{remark}
\newtheorem{rem}[defi]{Remark}

\AddToHook{env/definition/begin}{\crefalias{definition}{definition}}
\AddToHook{env/exa/begin}{\crefalias{definition}{example}}
\AddToHook{env/rk/begin}{\crefalias{definition}{remark}}
\AddToHook{env/proposition/begin}{\crefalias{definition}{proposition}}
\AddToHook{env/theorem/begin}{\crefalias{definition}{theorem}}
\AddToHook{env/corollary/begin}{\crefalias{definition}{corollary}}
\AddToHook{env/conjecture/begin}{\crefalias{definition}{conjecture}}

\usepackage{titling}

\usepackage{algorithm}
\usepackage{algpseudocode}
\algrenewcommand\algorithmicrequire{\textbf{Input:}}
\algrenewcommand\algorithmicensure{\textbf{Output:}}
\algtext*{EndIf}
\algtext*{EndWhile}
\algtext*{EndFor}
\algtext*{EndProcedure}

\title{On the integrality of some P-recursive sequences}
\author{Anastasia Matveeva}

\begin{document}

\hspace{0pt}
\vfill

\begin{center}

\includegraphics[width=0.25\textwidth]{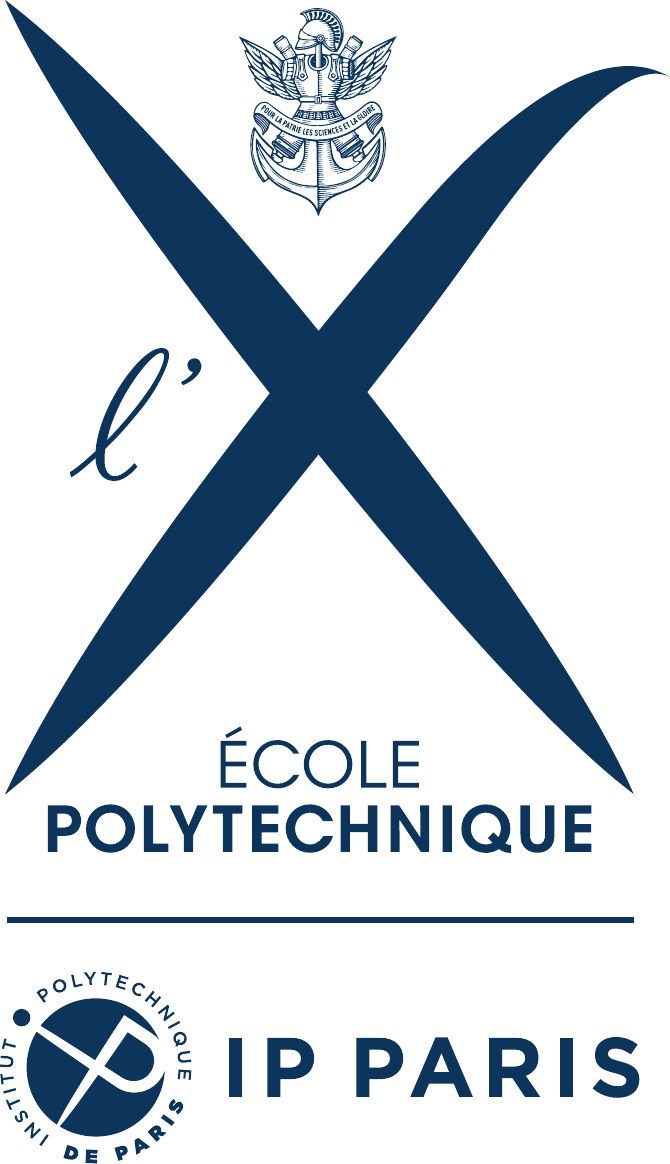}

\vspace*{2em}
{\large
\textbf{\'Ecole polytechnique}

\vspace*{1em}
\textit{LAB RESEARCH PROJECT}

\vspace*{3em}
{\Huge \textbf{On the integrality of some \\ P-recursive sequences}}
\vspace*{3em}

\textit{Student:}

\vspace*{1em}
Anastasia Matveeva, \'Ecole polytechnique

\vspace*{2em}
{\textit{Advisor:}}

\vspace*{1em}
Alin Bostan, Inria Saclay
}

\vspace*{2em}
\textit{February-May \ 2025}

\end{center}

\vfill
\hspace{0pt}
\newpage

\pagestyle{fancy}
\lhead{}
\rhead{\thesistitle}

\newpage
\section*{Acknowledgements}
I am deeply grateful to Alin Bostan for his guidance, support, and for introducing me to computer algebra in the first place. The many insightful conversations we had were truly what made this project possible.

My warm thanks to Florian Fürnsinn and Sergey Yurkevich for illuminating discussions and valuable comments, and also for making my holidays in Vienna all the more enjoyable.

Finally, I would like to thank my friends and family for listening --- with admirable patience --- to every detail of the writing process.
\newpage

\begin{abstract}
This project investigates the arithmetic nature of P-recursive sequences through the lens of their D-finite generating functions. Building on classical tools from differential algebra, we revisit the integrality criterion for Motzkin-type sequences due to Klazar and Luca, and propose a unified method for analysing global boundedness and algebraicity within a broader class of holonomic sequences. The central contribution is an algorithm that determines whether all, none, or a one-dimensional family of solutions to certain second-order recurrences are globally bounded. This approach generalizes earlier \emph{ad hoc} methods and applies successfully to several well-known sequences from the \href{https://oeis.org/}{OEIS}.
\end{abstract}

\begingroup
  \hypersetup{linkcolor=black}
  \tableofcontents
\endgroup

\newpage

\section{Introduction}

\begin{defi}\label{defi:prec}
A sequence $(u_n)_{n\in\mathbb{N}}$ of rational numbers is called \textit{P-recursive} (or \textit{P-finite}, or \textit{holonomic}) if it satisfies a linear (homogeneous) recurrence relation with coefficients in $\mathbb{Q}[n]$.
\end{defi}

\begin{defi}\label{defi:dfun}
A power series $f \in \mathbb{Q}[[x]]$ is called \textit{D-finite} (``differentially finite'', or ``holonomic'') if it satisfies a linear (homogeneous) differential equation with coefficients in $\mathbb{Q}[x]$.
\end{defi}

\begin{thm}(\cite[Theorem 1.5]{S80}) 
A power series is D-finite if and only if its coefficient sequence is P-recursive.
\end{thm}

\begin{defi}\label{defi:almint}
A power series $f(x) = \sum\limits_{n \geq 0}a_nx^n \in \mathbb{Q}[[x]]$ is called \textit{almost integral} if there exists $C \in \mathbb{Z}^*$ such that $f(Cx)-f(0) \in \mathbb{Z}[[x]]$.
Equivalently, if there exists $C \in \mathbb{Z}^* $ such that $C^n a_n \in \mathbb{Z}$ for all $n \geq 1$.
\end{defi}

\begin{defi}\label{defi:gb}
A power series $f \in \mathbb{Q}[[x]]$ is called \textit{globally bounded} if it is almost integral and has a nonzero radius of convergence in $\mathbb{C}$.
\end{defi}
P‑recursive sequences appear in many combinatorics problems, as well as in number theory and the theory of special functions. Whether a given P‑recursive sequence consists of integers only is a deceptively simple question, which is often approached with arguments tailored for particular sequences.
Our starting point is the family of Motzkin‑type sequences (defined properly in \cref{sec:Motzkin}), that is, of sequences that satisfy
\[(n+2)m_n = (2n+1)m_{n-1}+(3n-3)m_{n-2}.\]
with arbitrary initial values $m_0,m_1 \in \mathbb{Q}$. 

In \cref{sec:Motzkin}, we discuss the approach of Klazar and Luca to the integrality of Motzkin-type sequences and provide an algebraic argument for the integrality of Motzkin numbers. \cref{sec:arithmetic outlook} generalizes this to a full characterisation of algebraicity and global boundedness for the entire Motzkin-type family. In \cref{sec:alg GB}, we present an algorithm that decides these properties for an even larger family of second-order P-recursive sequences with polynomial coefficients of degree 1, under certain assumptions on the recurrence coefficients. The algorithm is then applied to several sequences from the \href{https://oeis.org/}{OEIS}, including large Schröder numbers and central trinomial coefficients. The integrality question is explored in \cref{sec:integrality}. Finally, in \cref{sec:small Apéry numbers}, we derive integrality criteria for small Apéry numbers, by adapting different approaches from the literature.

\section{Integrality analysis of Motzkin-type sequences}\label{sec:Motzkin}

A \textit{Motzkin-type sequence} $M(\mu, \lambda) = (m_n(\mu, \lambda))_n$ is defined in \cite{KL05} as a sequence of rational numbers that satisfies the recurrence:
\begin{equation} \label{eq:{1.1}} \tag{1.1}
(n+2)m_n(\mu, \lambda) = (2n+1)m_{n-1}(\mu, \lambda)+(3n-3)m_{n-2}(\mu, \lambda).
\end{equation}
with initial values $m_0(\mu, \lambda) = \mu$, $m_1(\mu, \lambda) = \lambda$. By abuse of notation, we sometimes write $m_n$ instead of $m_n(\mu, \lambda)$ when it is clear which initial values are considered.

In particular, $M(1,1)$ is the sequence of Motzkin numbers (\href{https://oeis.org/A001006}{A001006} in the OEIS). These numbers have a few interesting combinatorial interpretations, the most well-known being in terms of lattice paths. Motzkin numbers count the number of paths in the Cartesian plane from $(0,0)$ to $(n,0)$ that never dip below the horizontal axis and consist of up $(1,1)$, down $(1,-1)$, and level $(1,0)$ steps — known as \emph{Motzkin paths}. They are also closely related to Catalan numbers, which count the special case where level steps are not allowed (\emph{Dyck paths}).

In \cite{KL05} Klazar and Luca prove that $M(\mu, \lambda)$ is integral, that is, consists of integers only, if and only if $\mu, \lambda \in \mathbb{Z}$ and $\mu = \lambda$. In this section we provide a summary of their approach and discuss its limitations.

\subsection{Approach of Klazar and Luca} 
One starts by deducing the integrality of $M(\mu,\mu)=\mu M(1,1)$ for $\mu \in \mathbb{Z}$. This is nothing more than a straightforward corollary of the integrality of $M(1,1)$. If $\mu \notin\mathbb{Z}$, then $M(\mu,\mu)\notin\mathbb{Z}[[x]]$ as its first coefficient $\mu$ is already not an integer.
Regarding the case $M(\mu, \lambda)$ with $\mu \neq \lambda$, the following result holds, where $P(k)$ denotes the largest prime factor of $k \in \mathbb{Z}$ with the convention that $P(0) = P(\pm 1) = 1$.

\begin{thm}\label{thm:1.1} 
(\cite{KL05}, Theorem 1).
Let $M(\mu, \lambda) = (m_n(\mu, \lambda))_{n \geq 0} = (m_n)_{n \geq 0}$ be any Motzkin-type sequence of rational numbers with $m_0 = \mu \neq \lambda = m_1$. For any $n$ write $m_n = \frac{a_n}{b_n}$, where $a_n$ and $b_n$ are coprime integers with $b_n \geq 1$. Then $\lim \sup_{n \rightarrow \infty}$ $P(b_n) = \infty$.
\end{thm}
In other words, if $\mu \neq \lambda$, then for infinitely many primes $p$ there exists $n \in \mathbb{N}$ such that $m_n(\mu, \lambda)$ has a denominator divisible by $p$. Moreover, it can be shown that for $p$ large enough it is in fact $m_{p-2}(\mu, \lambda)$ that has a denominator divisible by $p$.

Below is a concise summary of the proof of \cref{thm:1.1}. We refer to \cite{KL05} for more details.

\noindent \textit{Proof strategy for \cref{thm:1.1}}.  
It suffices to exhibit a single pair $(\alpha,\beta)$ with $\alpha \neq \beta$ such that $M(\alpha,\beta)$ contains terms whose denominators are divisible by arbitrarily large primes, since the set of solutions to \eqref{eq:{1.1}} is two-dimensional. Indeed, every sequence $M(a,b)$ can be written as a linear combination of $M(1,1)$ and $M(\alpha,\beta)$. Since $M(1,1)$ is integral, $M(a,b)$ is almost integral for $a\neq b$ if and only if $M(\alpha,\beta)$ is.

To construct an easy-to-work-with example, consider the shifted generating power series $M(x) = \sum_{n \geq 0} m_n x^{n+2}$ and complete it to a power series $S(x) = a + b x + M(x)$ that satisfies an inhomogeneous linear differential equation of order 1. With a clever choice of $a, b, m_0, m_1$, this differential equation rewrites as
\begin{equation}\label{eq:{1.2}}\tag{1.2}
gS' - \tfrac{1}{2}g'S = g, \quad \text{where } g(x) = 1 - 2x - 3x^2.
\end{equation}
This equation admits the solution
\[
S(x) = \sqrt{g(x)} \int \frac{1}{\sqrt{g(x)}}\,dx,
\]
whose coefficients $m_k$ satisfy
\[
m_{k-2} = \sum_{n=0}^{k-1} \frac{c_n d_{k-n-1}}{k-n} = \frac{d_{k-1}}{k} + \sum_{n=1}^{k-1} \frac{c_n d_{k-n-1}}{k-n}, \quad \text{for all } k \geq 2,
\]
where $c_n$ and $d_n$ are the coefficients of the power series expansions of $\sqrt{g(x)} = 1+\sum\limits_{n\geq 1}c_nx^n$ and $1/\sqrt{g(x)} = 1+\sum\limits_{n\geq1} d_nx^n$, respectively. Both $(c_n)_n$ and $(d_n)_n$ are integer sequences.

The crucial point is that $d_{p-1}$ is not divisible by $p$ for any prime $p > 3$. This is established by analysing the explicit binomial sum representation of $1+\sum\limits_{n\geq 1} d_nx^n$. As a result, $m_{p-2}$ has a denominator divisible by $p$ for large $p$, and since $p$ is arbitrary, \cref{thm:1.1} follows. \qed

\subsection{Discussion}

A natural question to ask is:

\textit{To what extent can the approach described in \cite{KL05} be generalized and applied to other P-recursive sequences?}

\noindent \textbf{Comment 1.} The first claim made by Klazar and Luca is the integrality of Motzkin-type sequences of the form $M(\mu,\mu)$, $\mu \in \mathbb{Z}$. Although the argument appears to be direct, it relies heavily on the combinatorial properties of $M(1,1)$. Hence, the first limitation: in the general case the underlying combinatorial interpretation (if any) of a given P-recursive sequence is not known beforehand. This already arises two questions: 
\begin{enumerate}
\item[1.1] How to ``guess'' that the choice $m_0 = m_1 = 1$ results in an integral solution of \eqref{eq:{1.1}}?
\item[1.2] How to prove it? 
\end{enumerate}
Question 1.1 can be approached algorithmically with the approach presented in \cref{subsection: order 2}.

One way to answer question 1.2 is to consider the generating function $y(x) = \sum\limits_{n \geq 0} m_nx^n$. Its coefficient sequence is P-recursive thanks to \eqref{eq:{1.1}}, so we know that $y(x)$ is D-finite. Let us carry out some computations in Maple\footnote{Using the \href{https://perso.ens-lyon.fr/bruno.salvy/software/the-gfun-package/}{gfun} package, version 4.10.}

\begin{verbatim}
> rec := {(n + 2)m(n) = (2n + 1)m(n - 1) + (3n - 3)m(n - 2), m(0) = 1,  m(1)= 1} :
> with(gfun):
> rectodiffeq(rec, m(n), y(x));
\end{verbatim}
\[(-3x^2 - 3x + 2)y(x) + (-3x^3 - 2x^2 + x)\left(\frac{d}{dx} y(x)\right) - 2.\]
\begin{verbatim}
> dsolve({%, y(0) = 1}, y(x));
\end{verbatim}
\[y(x) = \frac{I\sqrt{3x - 1}\sqrt{x + 1} - x + 1}{2x^2}.\]

In the procedure shown above we computed the differential equation satisfied by the generating function of Motzkin numbers with the \textit{rectodiffeq} command, and solved it using \textit{dsolve}. The initial condition $y(0) = 1$ comes from the fact that $y(0) = m_0 = 1$. Equivalently:
\[y(x) = \frac{1-x-\sqrt{1-2x-3x^2}}{2x^2}.\]
and one recognizes the closed form for the generating function of Motzkin numbers. 

It is now easy to see why all $m_n \in \mathbb{Z}$: let $\sum\limits_{n \geq 0}a_nx^n \in \mathbb{Q}[[x]]$ be the power series expansion of $\sqrt{1-2x-3x^2}$. Then
\[\left(\sum\limits_{n \geq 0}a_nx^n\right)^2 = \sum\limits_{n \geq 0} x^n \sum\limits_{k=0}^{n} a_{k}a_{n-k} = 1-2x-3x^2.\]
For all $k\in\mathbb{N}$ and $ P(x)\in \mathbb{Q}[[x]]$ let $[x^k]P(x)$ denote the coefficient of $x^k$ in the power series expansion of $P(x)$. Matching the first few coefficients gives
\begin{align*}
[x^0] \sum\limits_{n \geq 0} x^n \sum\limits_{k=0}^{n} a_{k}a_{n-k} & = a_0 = 1, \\
[x^1] \sum\limits_{n \geq 0} x^n \sum\limits_{k=0}^{n} a_{k}a_{n-k} & = a_0a_1 + a_1a_0 = -2 \implies a_1 = -1, \\
[x^2] \sum\limits_{n \geq 0} x^n \sum\limits_{k=0}^{n} a_{k}a_{n-k}  & = a_0a_2 + a_1a_1 + a_2a_0 = -3 \implies a_2 = -2.
\end{align*}
Hence \[\frac{1-x-\sqrt{1-2x-3x^2}}{2x^2} = \frac{1-x-1+x-\sum_{n \geq 2}a_nx^n}{2x^2} = \frac{1}{2}\sum\limits_{n \geq 2}a_nx^{n-2}.\] It remains to prove that $a_n$ is an even integer for $n \geq 2$.

\begin{proof} 
We proceed by induction. The statement holds for $n = 2$ as $a_2 = -2$. Assume $a_k \in 2\mathbb{Z}$ for all $k \leq N$ for some $N \geq 2$. Then 
\[0 = [x^{N+1}] \sum\limits_{n \geq 0} x^n \sum\limits_{k=0}^{n} a_{k}a_{n-k} = 2a_0a_{N+1} + 2a_1a_N + \sum\limits_{k=2}^{N-1} a_{k}a_{N+1-k} \implies a_{N+1} = a_N - \frac{1}{2} \sum\limits_{k=2}^{N-1} a_{k}a_{N+1-k}.\]
By the inductive assumption, $a_k$ is even for $2\leq k \leq N$, so $2|a_N$ and $4|a_ka_{N-k}$ for $2\leq k \leq N-1$. Hence, 
\[2\  \text{ divides }\  a_N - \frac{1}{2} \sum\limits_{k=2}^{N-1} a_{k}a_{N+1-k} = a_{N+1}.\]
\end{proof}

\begin{rem}
The above reasoning could have been made shorter by appealing to known identities for Motzkin numbers. For example, one could argue that 
\[m(x) = \frac{1-x-\sqrt{1-2x-3x^2}}{2x^2}\]
is a standard way of expressing terms of $M(1,1)$ and therefore the computation by hand (or within a computer algebra system) is redundant. However, most of such results utilize implicitly the connections with combinatorial objects and this is precisely what we are trying to avoid for the sake of generality.
\end{rem}

\noindent \textbf{Comment 2.} It is noteworthy that the proof of \cref{thm:1.1} depends significantly on the structure of the differential equation \eqref{eq:{1.2}}. As a result, even slight changes to the recurrence parameters require adjustments to the method.

In view of \cref{defi:gb}, and using the fact that $M(\mu,\mu)$ is globally bounded (because it belongs to $\mu\mathbb{Z}[[x]]$) for any $\mu \in \mathbb{Q}$, we can reformulate \cref{thm:1.1} as follows:

\begin{thm}\label{thm:1.2}
A Motzkin-type sequence $M(\mu, \lambda)$ with $\mu,\lambda \in \mathbb{Q}$ is globally bounded if and only if $\mu = \lambda$.
\end{thm}
Let us now invoke an important result, first stated by Eisenstein in 1852 (\cite{Eis52}) and subsequently proved by Heine (\cite{Heine53}, \cite{Heine54}):

\begin{thm}\label{thm:Eisenstein} (Eisenstein's theorem). If $f(x) \in \mathbb{Q}[[x]]$ is algebraic, then it is globally bounded.
\end{thm}

Using the fact that the generating function of Motzkin numbers (and, consequently, of any $M(\mu, \mu)$ by linearity) is algebraic, we obtain the following corollary:
\begin{cor}\label{cor:2.1} (of \cref{thm:1.2} and \cref{thm:Eisenstein})
The generating function of a Motzkin-type sequence $M(\mu, \lambda)$ is algebraic if and only if $\mu = \lambda$.
\end{cor}
\section{Arithmetic outlook}\label{sec:arithmetic outlook}

We recall that Motzkin-type sequences satisfy the recurrence
\begin{equation}\label{eq:{2.1}} \tag{2.1}
(n+2)m_n = (2n+1)m_{n-1}+(3n-3)m_{n-2}, \quad n \geq 2
\end{equation}
and are uniquely defined by a pair of initial values $m_0,m_1$. We will also write $m_n(\lambda,\mu)$ for the $n$th term of a Motzkin-type sequence with $m_0=\mu, m_1=\lambda$ when it is necessary to specify the initial values.

In what follows, we will also make use of the notion of diagonals of multivariate power series:
\begin{defi}\label{defi:diag}
Let 
\[R(x_1,\ldots,x_k) = \sum\limits_{n_1,\ldots,n_k\geq0} c(n_1,\ldots,n_k)x_1^{n_1}\ldots x_k^{n_k} \in \mathbb{Q}[[x_1,\ldots,x_k]].\]
The \textit{diagonal} of $R$ is a univariate power series defined by
\[\text{Diag}(R) = \sum\limits_{n\geq0} c(n,\ldots,n)t^n \in \mathbb{Q}[[t]].\]
\end{defi}

\begin{prop}\label{prop:3.1}
(Furstenberg, \cite{Fur67}).
Any algebraic power series is the diagonal of a rational function in two variables.
\end{prop}
The converse of \cref{prop:3.1}, known as Pólya's theorem \cite{Polya21}, is also true. 
\begin{prop}\label{prop:3.2}(\cite[Proposition 5, p.~49]{Chr90})
Diagonals of rational functions are globally bounded.
\end{prop}

Set $M(x) = \sum\limits_{n\geq0} m_nx^n$. The next result describes several equivalent conditions involving the arithmetic nature of $M(x)$, and the parameters defining the sequence $(m_n)_n$.

\begin{thm}\label{thm:5.1}
The following statements are equivalent:
\begin{enumerate}
    \item $M(x)$ is algebraic.
    \item $M(x)$ is globally bounded. 
    \item There exists $b\in\mathbb{Z}^*$ such that $m_n\in\frac{1}{b}\mathbb{Z}$ for all $n \geq 0$.
    \item $m_0 = m_1$.
    \item $M(x)$ is a diagonal of some bivariate rational function.
\end{enumerate}
\end{thm}
\begin{proof} \
\begin{enumerate}
\item[$1\implies2$.] By Eisenstein's theorem.
\item[$2\implies4$.] By the contrapositive of \cref{thm:1.2}.
\item[$1\iff4$.] By \cref{cor:2.1}.
\item[$4\implies3$.] By the fact that the sequence of Motzkin numbers, i.e. $(m_n(1,1))_n$,  is integral and $m_n(\frac{a}{b},\frac{a}{b}) = \frac{a}{b} m_n(1,1)$ for all $n \geq 0$ and all $a,b \in \mathbb{Z}$.
\item[$3\implies2$.] Trivial.
\item[$1\implies5$.] By \cref{prop:3.1}.
\item[$5\implies2$.] By \cref{prop:3.2}.
\end{enumerate}
\end{proof}

The figure below visualizes the implications in \cref{thm:5.1}:
\begin{figure}[H]
\centering
\includegraphics[scale=0.35]{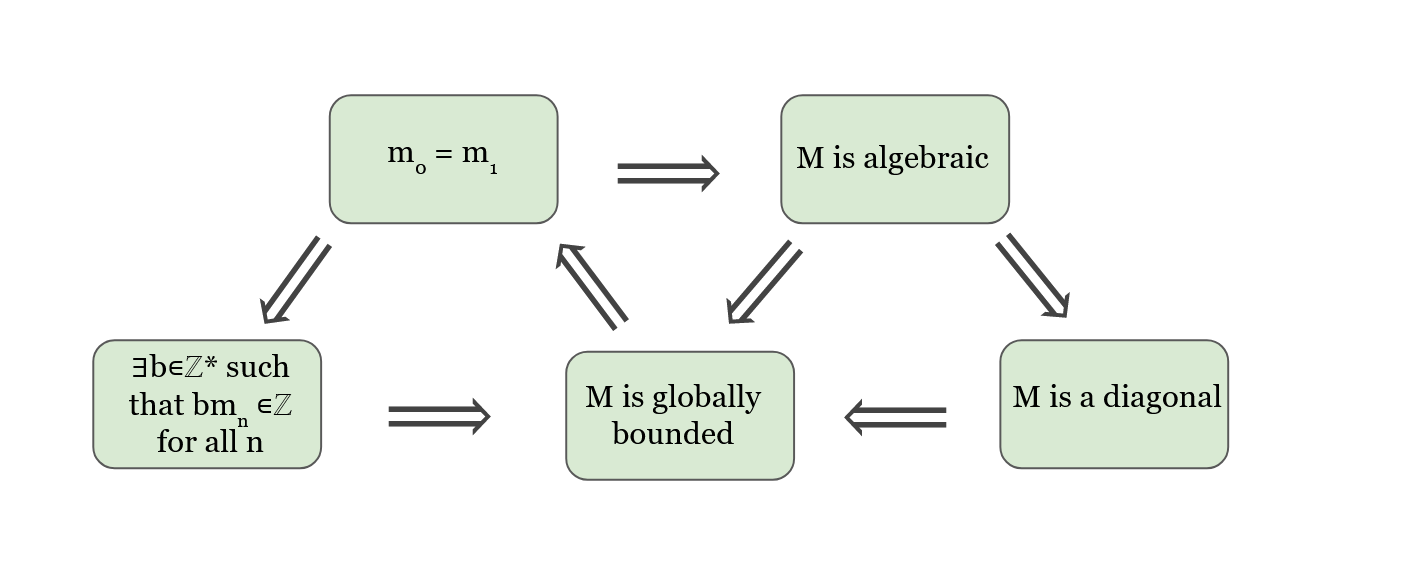}
\caption{}
\end{figure}

Such a result is surprising as it establishes, in the case of this specific recurrence, the converse of Eisenstein's theorem (the implication 1. $\implies$ 2.). Of course, this is no longer true in the general setting, and there are many examples of transcendental functions with almost integral, or even integral power series expansions. For instance, the generating function of the squared central binomial coefficients lives in $\mathbb{Z}[[x]]$, but at the same time is transcendental (\cite[p.~271--272]{Fur67}). It writes as
\[_{2}F_{1}\!\left[\!\begin{array}{c}\ \frac{1}{2} \ \frac{1}{2}\\ \ 1
\end{array}
; \, 16x
\right] = \sum\limits_{n\geq0} \binom{2n}{n}^2x^n,\]
where ${}_2F_1$ denotes the Gaussian hypergeometric function defined as 
\[_{2}F_{1}\!\left[\!\begin{array}{c}\ a \ \ b\\ \ c
\end{array}
; \, x
\right] = \sum\limits_{n=0}^{\infty} \frac{(a)_n(b)_n}{(c)_n} \frac{x^n}{n!},\]
and $(q)_n$ is the rising factorial (also known as \emph{Pochhammer symbol}):
\[
(q)_n = 
\begin{cases}
1 & \text{if } n = 0, \\
q(q+1)\ldots(q+n-1) & \text{if } n > 0.
\end{cases}
\]

There is no contradiction with the general falsity of the converse. At a closer look, the equivalence of 1. and 2. relies on an intermediate step: namely the condition $m_0=m_1$. The phenomenon is captured by the chain of implications $1 \implies 2 \implies 4$. 

Another curious observation is that the implication $2 \implies 5$ entails Christol's conjecture (\cite{Chr90}, Conjecture 4, p.~55) for P-recursive sequences that are solutions of \eqref{eq:{2.1}}.

\begin{conj}(\cite{Chr90}) Every D-finite globally bounded function is the diagonal of a rational function.
\end{conj}
It is worth mentioning that \cref{thm:5.1} is not easily generalizable to other linear recurrences.
Even in the very specific case of second-order recurrences, the implications $2 \implies 1$ and $3 \implies 1$ do not hold if we drop the linearity of polynomial coefficients. As a counterexample one can take the recurrence for big Apéry numbers (\href{https://oeis.org/A005259}{OEIS A005259}):
\begin{equation}\label{eq:{2.2}}\tag{2.2}
n^3 A_n = (34n^3-51n^2+27n-5)A_{n-1} - (n-1)^3A_{n-2}.
\end{equation}
The sequence solving the recurrence above with $A_0=1, A_1=5$ appears in Apéry's famous proof of the irrationality of $\zeta(3)$ (\cite{Ap79}), and is known to be integral. Integrality can be deduced, for example, from the binomial sums representation of $(A_n)_n$:
\[A_n = \sum\limits_{k=0}^{n} \binom{n}{k}^2 \binom{n+k}{k}^2, \quad  n \in \mathbb{N}.\]

There exist several proofs of the fact that $A_0 = 1, A_1 = 5$ is the only pair of initial values (up to multiplication by an integer) that produces an integer sequence (\cite{C84}, \cite{AJ91}). However, the generating function $\sum\limits_{n\geq0}A_nx^n$ of any sequence that solves \eqref{eq:{2.2}} is transcendental (one reason for this is asymptotics incompatible with algebraicity, the so-called \emph{Flajolet criterion}).

A similar argument applies to small Apéry numbers (\href{https://oeis.org/A005258}{OEIS A005258}): the generating function of any sequence produced by the recurrence for small Apéry numbers is transcendental, while there do exist almost integral and even integral solutions of this recurrence. A detailed analysis of such cases is provided in \cref{sec:small Apéry numbers}.

\section{Bridging global boundedness and algebraicity} \label{sec:alg GB}

One possible generalization of \cref{thm:5.1} is the following one, which we state as a conjecture:

\begin{conj}\label{conj:algiffgb}
Let $(s_n)_n$ be a sequence of rational numbers that satisfies a linear homogeneous recurrence relation with polynomial coefficients of degree 1:
\[\sum\limits_{k=0}^{d}(a_kn+b_k)s_{n-k} = 0, \quad a_i \neq 0 \text{ for } i =0,\ldots,d \quad d \in \mathbb{N}.\]
Let $S(x) = \sum\limits_{n\geq0} s_nx^n$ be its generating function. The following statements are equivalent:
\begin{enumerate}
    \item $S(x)$ is algebraic.
    \item $S(x)$ is a diagonal.
    \item $S(x)$ is globally bounded.
\end{enumerate}
\end{conj}

As before, 

$1 \implies 2$ by \cref{prop:3.1}, 

$2 \implies 3$ by \cref{prop:3.2}.

The interesting direction is $3 \implies 1$. We prove it for first-order recurrences in \cref{subsection: order 1} and for a specific subclass of second-order recurrences in \cref{subsection: order 2}. The general case of \cref{conj:algiffgb} is currently unsolved.

\subsection{Hypergeometric case}\label{subsection: order 1}
The most basic setting of \cref{conj:algiffgb} is the case $d=1$. 
It defines a certain subclass of hypergeometric functions, that is, of power series whose coefficient sequences satisfy first-order linear recurrences with polynomial coefficients. Specifically, we consider $(s_n)_n$ such that 
\[(a_0n+b_0)s_n+(a_1n+b_1)s_{n-1}=0, \quad  a_0,a_1 \in \mathbb{Z}^*, \quad b_0,b_1 \in \mathbb{Z}, \quad \frac{b_0}{a_0} \not\in -\mathbb{N}^*.\]
Or, equivalently
\begin{equation}\label{eq:{3.1}}\tag{3.1}
(n+b_0)s_n+(a_1n+b_1)s_{n-1}=0, \quad a_1 \in \mathbb{Q}^*, \quad b_0 \in \mathbb{Q}\setminus -\mathbb{N}^*, \quad b_1 \in \mathbb{Q},
\end{equation}
where the condition $\frac{b_0}{a_0} \not\in -\mathbb{N}^*$ comes from the fact that if $a_0n+b_0 = 0$ for some $n \in \mathbb{N}^*$, then $s_n$ is undefined unless $s_k=0$ for all $k \in \mathbb{N}$.

Let $S(x) = \sum\limits_{n\geq0} s_nx^n$ denote the generating function of $(s_n)_n$.

The usual scheme of transforming \eqref{eq:{3.1}} into a differential equation and then solving it yields:
\begin{equation}\label{eq:{3.2}}\tag{3.2}
S(x) = x^{-b_0}(a_1x + 1)^{-1+b_0 - \frac{b_1}{a_1}} \left(s_0b_0\int x^{-1 + b_0}(a_1x + 1)^{\frac{-a_1b_0 + b_1}{a_1}}dx + c_1\right).
\end{equation}
We can clearly see that $S(x)$ is a combination of algebraic functions and a primitive of an algebraic function 
\[x^{-1 + b_0}(a_1x + 1)^{\frac{-a_1b_0 + b_1}{a_1}}.\] Fortunately, the problem of integrating algebraic functions has been extensively studied and goes back to the work of Liouville \cite[Chapter IX]{Lu12}. In our analysis we make use of the following (rather involved) theorem due to André:

\begin{thm}\label{thm:André}
(André, \cite[p.~149]{And89}).
Let $y$ be a power series. Then $y$ is algebraic if and only if $y$ is globally bounded and $\frac{dy}{dx}$ is algebraic.
\end{thm}

\begin{cor} (of \cref{thm:André} and \cref{thm:Eisenstein}).
The primitive of an algebraic power series is
globally bounded if and only if it is algebraic. 
\end{cor}

Therefore, $S(x)$ is algebraic if and only if it is globally bounded, which immediately implies \cref{conj:algiffgb} for this case $d=1$.

The next theorem characterises the conditions under which $S(x)$ is algebraic and globally bounded, noting that these properties are equivalent by \cref{conj:algiffgb}.

\begin{thm}\label{thm:gbhyper}
Let $(s_n)_n$ be a sequence of rational numbers that satisfies \eqref{eq:{3.1}}. Then the generating function $S(x) = \sum\limits_{n\geq0} s_nx^n$ is globally bounded (and algebraic) if and only if one of the following conditions holds:
\begin{enumerate}
  \item $b_0\in\mathbb N,\ \frac{b_1}{a_1}\notin\mathbb Z$, or 
  \item $b_0\notin\mathbb Z,\ b_0-\frac{b_1}{a_1}\in-\mathbb N$, or 
  \item $b_0,b_0-\frac{b_1}{a_1}\notin\mathbb Z,\ \frac{b_1}{a_1}+1\in-\mathbb N$, or 
  \item $b_0,\frac{b_1}{a_1}\in\mathbb Z,\  \frac{b_1}{a_1}<0<b_0, \ \frac{b_1}{a_1}\le -1$, or
  \item $b_0,\frac{b_1}{a_1}\in\mathbb Z,\  0 < b_0\leq\frac{b_1}{a_1}$.
\end{enumerate}
\end{thm}

\begin{proof}
$S(x)$ can be expressed from \eqref{eq:{3.2}} as
\begin{equation}\label{eq:{3.3}}\tag{3.3}
S(x) = x^{-b_0}(a_1x + 1)^{-1+b_0 - \frac{b_1}{a_1}} \left(s_0x^{b_0} {}_{2}F_{1}\!\left[\!\begin{array}{c} b_0 \quad b_0-\frac{b_1}{a_1}\\1+b_0
\end{array}
; \, -a_1x
\right] + c_1\right).
\end{equation}

Therefore, $S(x)$ is algebraic precisely when ${}_{2}F_{1}\!\left[\!\begin{array}{c} b_0 \quad b_0-\frac{b_1}{a_1}\\1+b_0
\end{array}
; \, -a_1x
\right]$ is.

An algorithm to determine the algebraicity of hypergeometric functions with arbitrary parameters is described by Fürnsinn and Yurkevich in \cite{FY24}. In the same paper, a useful criterion for deciding the algebraicity of Gaussian hypergeometric functions is presented:

\begin{cor}\label{cor:gaussalg}
(\cite{FY24}, Corollary 3.7) \\
The Gaussian hypergeometric function ${}_{2}F_{1}\!\bigl([\alpha,\beta],[\,\alpha+k];x\bigr)$ for $k\in\mathbb Z$,
is algebraic if and only if either $k\le 0$ or:
\begin{multicols}{2}
\begin{enumerate}
  \item $\alpha\in\mathbb Z,\ \beta\notin\mathbb Z$, or 
  \item $\alpha\notin\mathbb Z,\ \beta\in-\mathbb N$, or 
  \item $\alpha,\beta\notin\mathbb Z,\ \beta-\alpha-k\in\mathbb N$, or 
  \item $\alpha,\beta\in\mathbb Z,\ \alpha<\beta\le 0,\ \beta-\alpha\ge k$, or 
  \item $\alpha,\beta\in\mathbb Z,\ 0<\alpha<\beta,\ \beta-\alpha\ge k$, or
  \item $\alpha,\beta\in\mathbb Z,\ \beta\le 0<\alpha$.
\end{enumerate}
\end{multicols}
\end{cor}

In our case of ${}_{2}F_{1}\left[\!\begin{array}{c} b_0 \quad b_0-\frac{b_1}{a_1}\\1+b_0
\end{array}
; \, -a_1x
\right]$, the values are $\alpha = b_0, \beta = b_0-\frac{b_1}{a_1}, k = 1$. By substituting them into \cref{cor:gaussalg}, we obtain the necessary and sufficient conditions for the algebraicity of $S(x)$, captured in \cref{thm:gbhyper}.

Observe that case 4. in \cref{cor:gaussalg} is not possible for the choice $\alpha = b_0, \beta = b_0-\frac{b_1}{a_1}, k = 1$.  Indeed, the three requirements in that line become
\[b_0,\frac{b_1}{a_1}\in\mathbb Z,\quad  b_0<b_0-\frac{b_1}{a_1}\le0, \quad -\frac{b_1}{a_1}\ge 1.\]
which simplify to
\[b_0,\frac{b_1}{a_1}\in\mathbb Z,\  b_0\le \frac{b_1}{a_1}\le -1.\]
This forces $b_0$ to be a nonpositive integer, contradicting our assumption $b_0 \in \mathbb{Q}\setminus -\mathbb{N}^*$. Hence case 4. is ruled out for the parameters that occur in \cref{thm:gbhyper}.
\end{proof}

\subsection{Second-order recurrences}\label{subsection: order 2}
Let $(s_n)_n$ be a P-recursive sequence of rational numbers that satisfies a second-order recurrence relation:
\begin{equation} \label{eq:{4.1}}\tag{4.1}
(n+b_0)s_n+(a_1n+b_1)s_{n-1} + (a_2n+b_2)s_{n-2}=0
\end{equation}
with $a_1,a_2 \in \mathbb{Q}^*, \quad b_0,b_1,b_2 \in \mathbb{Q}, \quad b_0 \not\in -\mathbb{N} \setminus \{0,-1\}$. \\ 
Let $S(x) = \sum\limits_{n \geq 0} s_nx^n$ be the generating function of $(s_n)_n$.
In the general case, $S(x)$ takes the following form:
\begin{verbatim}
> rec := {(n + b_0)*s(n) + (a_1*n + b_1)*s(n-1) + (a_2*n + b_2)*s(n-2), 
         s(0) = s_0, s(1) = s_1}:
> dsolve(rectodiffeq(rec, s(n), S(x)), S(x));
\end{verbatim}
\begin{align*}
S(x) &= (a_2 x^2 + a_1 x + 1)^{-\frac{b_2}{2 a_2} - 1 + \frac{b_0}{2}} 
\Bigg(\int 
\left( a_1 s_0 x + b_0 s_1 x + b_1 s_0 x + b_0 s_0 + s_1 x \right) 
(a_2 x^2 + a_1 x + 1)^{-\frac{a_2b_0-b_2}{2a_2}}x^{b_0-1} \\
&\quad e^{\frac{\operatorname{artanh} \left( \frac{2 a_2 x + a_1}{\sqrt{a_1^2 - 4 a_2}} \right)
\left(a_1a_2b_0 + a_1 b_2 - 2 a_2 b_1 \right)}{
\sqrt{a_1^2 - 4 a_2 }a_2}} dx 
+ c_1 \Bigg)x^{-b_0}e^{-\frac{\operatorname{artanh} \left( \frac{2 a_2 x + a_1}{\sqrt{a_1^2 - 4 a_2}} \right)
\left(a_1a_2b_0 + a_1 b_2 - 2 a_2 b_1 \right)}{
\sqrt{a_1^2 - 4 a_2 }a_2}},
\end{align*}
where $c_1$ is some constant that comes from solving a differential equation with unspecified initial conditions. Supposedly, if $e$ does not vanish in the equation above, $S(x)$ is not globally bounded. However, this question remains open for now.

\begin{center}
\textbf{From now on we shall restrict ourselves to the case $\mathbf{b_2 = \frac{2a_2b_1-a_1a_2b_0}{a_1}}$, setting the exponent of e to zero.}
\end{center}
\begin{verbatim}
> simplify(subs(b_2 = (2*a_2*b_1 - a_1*a_2*b_0)/a_1, %));
\end{verbatim}
\begin{align*}
S(x) &= (a_2 x^2 + a_1 x + 1)^{-\frac{(b_0-1)a_1-b_1}{a_1}} \\&
\Bigg(\int x^{b_0}(a_2 x^2 + a_1 x + 1)^{-b_0+\frac{b_1}{a_1}}
\left( a_1 s_0 + b_0 s_1 + b_1 s_0 + \frac{b_0 s_0}{x} + s_1 \right)  dx 
+ c_1 \Bigg)x^{-b_0}.
\end{align*}

Since the integrand in the code snippet above is algebraic, André’s theorem ensures that global boundedness and algebraicity of $S(x)$ are equivalent in this context. Hence, \cref{conj:algiffgb} is proved for the case $b_2 = \frac{2a_2b_1-a_1a_2b_0}{a_1}$. 
\begin{center}
\textbf{In light of this result, we treat algebraicity and global boundedness of ${S(x)}$ as interchangeable throughout this section.}
\end{center}

Given a recurrence of type \eqref{eq:{4.1}}, a sequence $(s_n)_n$ is uniquely defined by a pair $(s_0,s_1)$. The algebraic nature of $S(x)$ falls therefore into one of three cases:
\begin{enumerate}[label=\textbf{(C\arabic*)}, ref=C\arabic*]
  \item \label{case:alg} $S(x)$ is algebraic for all $(s_0,s_1) \in \mathbb{Q}^2$.
  \item \label{case:trans} $S(x)$ is transcendental for all $(s_0,s_1) \in \mathbb{Q}^2 \setminus \{(0,0)\}$.
  \item \label{case:line} The set of pairs $(s_0,s_1)$ such that $S(x)$ is algebraic forms a one-dimensional subspace of~$\mathbb{Q}^2$.
\end{enumerate}
We now present an algorithm that determines which of the three cases holds for a given recurrence of type \eqref{eq:{4.1}} with 
\[b_2 = \frac{2a_2b_1-a_1a_2b_0}{a_1}, \; b_0 \in \mathbb{N}.\]
If $S(x)$ is algebraic for all initial conditions or for none (\cref{case:alg} and \cref{case:trans}), the algorithm will detect this directly. In \cref{case:line}, where algebraicity occurs only along a one-dimensional subspace of $\mathbb{Q}^2$, the algorithm computes a nonzero rational pair $(s_0,s_1)$ that generates all such algebraic solutions, up to scalar multiplication.
Our approach involves a case-by-case analysis based on the recurrence structure. 

Firstly, it might happen that $a_1^2-4a_2 = 0$. That is, $1+a_1x+a_2x^2$ has a double root $a = -\frac{a_1}{2a_2}$. This case is simple\footnote{Assuming your institution provides a Mathematica license.}. One computes:
\begin{align*}
S(x) &= x^{-b_0} (1 + a x)^{2(b_0 - 1) - \frac{b_1}{a}} \bigg( 
  c_1 + x^{b_0} \bigg( 
    s_0 (1 + a x)^{1 - 2b_0 + \frac{b_1}{a}} \\
    &\quad + (a s_0 + s_1) \, {}_2F_1\left(
      [1 + b_0,\;
      2b_0 - \frac{b_1}{a}],\;
      [2 + b_0];\;
      -a x
    \right) 
  \bigg) 
\bigg).
\end{align*}
It is immediate that $S(x)$ has a nontrivial algebraic solution when $as_0+s_1=0 \ (\iff -\frac{a_1}{2a_2}s_0+s_1=0 \iff a_1s_0=2a_2s_1)$. For example, one can take $(s_0,s_1) = (2a_2, a_1)$. Moreover, by \cref{cor:gaussalg},
$S(x)$ is algebraic for all $(s_0,s_1) \in \mathbb{Q}^2$ if and only if:
\begin{enumerate}
  \item $b_0\in\mathbb Z,\ \frac{b_1}{a}\notin\mathbb Z$, or 
  \item $b_0\notin\mathbb Z,\ 2b_0-\frac{b_1}{a}\in-\mathbb N$, or 
  \item $b_0,2b_0-\frac{b_1}{a}\notin\mathbb Z,\ b_0-\frac{b_1}{a}-2\in\mathbb N$, or 
  \item $b_0,\frac{b_1}{a}\in\mathbb Z,\ b_0>-1,\ b_0-\frac{b_1}{a}\ge 2$, or
  \item $b_0,\frac{b_1}{a}\in\mathbb Z,\  b_0>-1, \ 2b_0-\frac{b_1}{a}\le 0$.
\end{enumerate}
If we insist that $b_0\in\mathbb{N}$ and substitute back $a = -\frac{a_1}{2a_2}$, the list above becomes
\begin{enumerate}
  \item $\frac{2a_2b_1}{a_1}\notin\mathbb Z$, or 
  \item $\frac{2a_2b_1}{a_1}\in\mathbb Z,\ b_0+\frac{2a_2b_1}{a_1}\ge 2$, or
  \item $\frac{2a_2b_1}{a_1}\in\mathbb Z,\ \ b_0 +\frac{a_2b_1}{a_1} \le 0$.
\end{enumerate}

We now turn to the case $a_1^2-4a_2 \neq 0$ and study the integral
\begin{equation}\label{eq:{4.2}}\tag{4.2}
\int{\left(x^{b_0-1}b_0s_0 + x^{b_0}(a_1s_0+b_0s_1+b_1s_0+s_1)\right)(1+a_1x+a_2x^2)^{-b_0+\frac{b_1}{a_1}}dx},
\end{equation}
and ask when it is algebraic. 

\begin{center}
\textbf{From now on we also assume $b_0\in\mathbb{N}^*$, leaving the general case for future work.} 
\end{center}

The key idea is to represent \eqref{eq:{4.2}} as a linear combination of two integrals of the type 
\begin{equation}\label{eq:{4.3}}\tag{4.3}
\int{x^n (1+a_1x+a_2x^2)^q dx}, \quad a_1,a_2 \in \mathbb{Q}^*, \ n \in \mathbb{N}, \ q \in \mathbb{Q}.
\end{equation}

Indeed, let \[I_1 = \int{x^{b_0-1}(1+a_1x+a_2x^2)^{-b_0+\frac{b_1}{a_1}}dx}, \quad I_2 = \int{x^{b_0}(1+a_1x+a_2x^2)^{-b_0+\frac{b_1}{a_1}}dx}.\]
Then
\[\int{\left(x^{b_0}(a_1s_0+b_0s_1+b_1s_0+s_1) + x^{b_0-1}b_0s_0\right)(1+a_1x+a_2x^2)^{-b_0+\frac{b_1}{a_1}}dx} = b_0s_0 I_1 + (a_1s_0+b_0s_1+b_1s_0+s_1)I_2.\]
Observe that
\[b_0s_0=0 \iff s_0 = 0 \ \text{ since } b_0\ge1\]
\[a_1s_0+b_0s_1+b_1s_0+s_1 = 0 \iff (b_0+1)s_1 = -s_0(a_1+b_1).\]

The algorithm proceeds as follows:
\begin{enumerate}
\item Decide algebraicity of \( I_1 \) and \( I_2 \).
\item Proceed according to the table below:
\end{enumerate}
\begin{center}
\renewcommand{\arraystretch}{1.5}
\begin{tabular}{|c|c|p{8.5cm}|}
\hline
\textbf{\( I_1 \)} & \textbf{\( I_2 \)} & \textbf{Conclusion} \\
\hline
Algebraic & Algebraic & \cref{case:alg}. \\
\hline
Algebraic & Transcendental & \cref{case:line}, return $(b_0+1,-a_1-b_1)$. \\
\hline
Transcendental & Algebraic & \cref{case:line}, return $(0,1)$. \\
\hline
Transcendental & Transcendental & Check whether the combination
\[
b_0s_0 I_1 + (a_1s_0 + b_0s_1 + b_1s_0 + s_1)I_2
\]
is algebraic for some \( (s_0, s_1) \in \mathbb{Q}^2 \setminus  \{(0,0)\} \).  
If so, then \cref{case:line}, return $(s_0,s_1)$. Otherwise, \cref{case:trans}. \\
\hline
\end{tabular}
\end{center}

\subsubsection{Deciding algebraicity of $I_1,I_2$.}
\begin{figure}[H]
\centering
\includegraphics[scale=0.5]{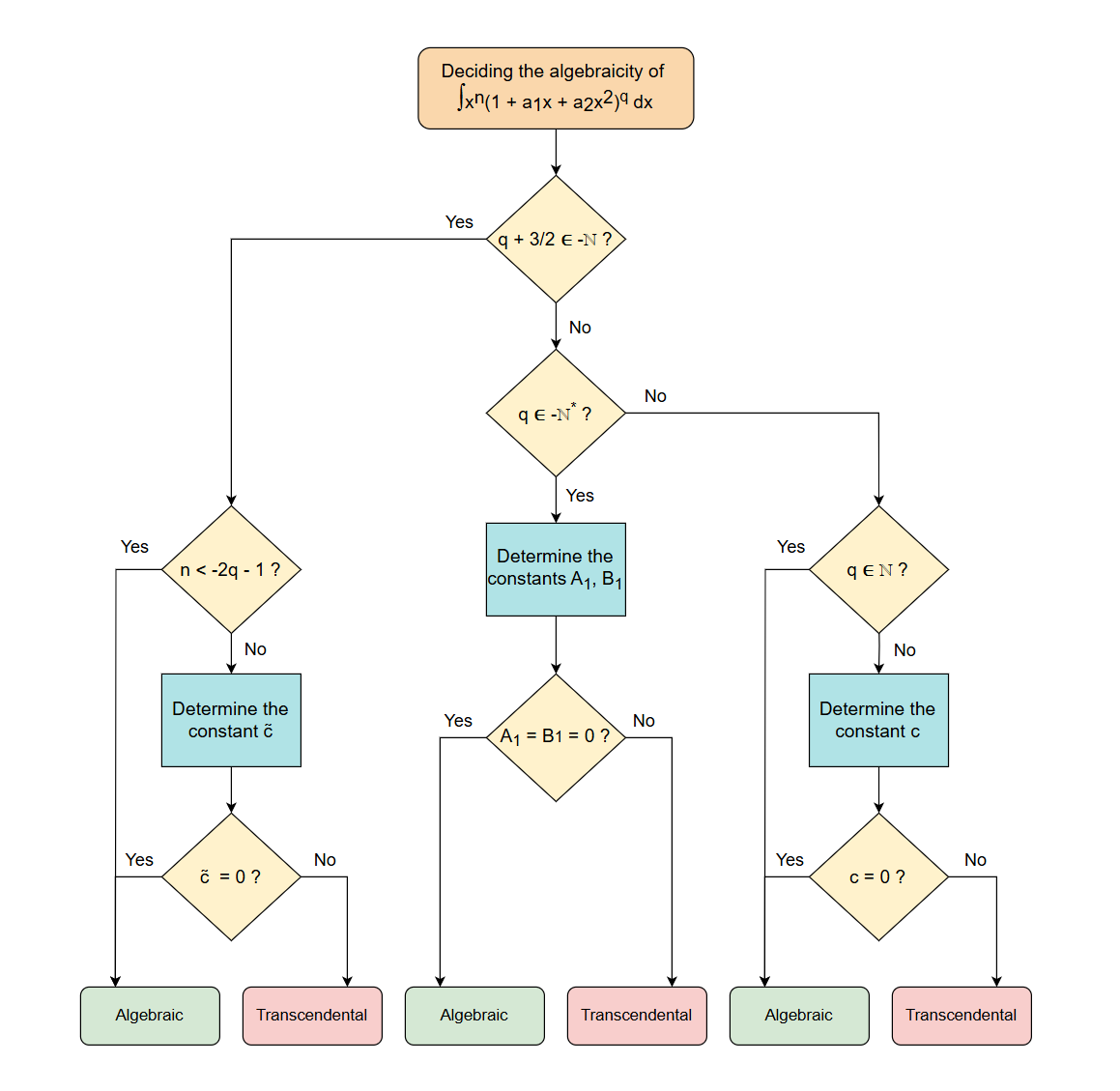}
\caption{Deciding whether $\int{x^n (1+a_1x+a_2x^2)^q dx}$, such that $a_1,a_2 \in \mathbb{Q}^*, n \in \mathbb{N}, q \in \mathbb{Q}$, is algebraic or transcendental. The constants $c, \tilde{c}, (A_1, B_1)$ come from \eqref{eq:{4.4}}, \eqref{eq:{4.7}}, \eqref{eq:{4.8}}, respectively.}
\label{fig:alg}
\end{figure}
Consider an integral of the type \eqref{eq:{4.3}}. There are several cases:
\begin{enumerate}
\item $\mathbf{n = 0}$. \\
In order to study the integral
\[\int{(1+a_1x+a_2x^2)^qdx}.\]
we distinguish 2 subcases:
\begin{itemize}
\item[1.2] Assume $\mathbf{a_1^2-4a_2>0}$. 

If $q=-1,$
\[\int{(1+a_1x+a_2x^2)^{-1} dx} = -\frac{2\operatorname{artanh} \left
(\frac{2a_2x + a_1}{\sqrt{a_1^2 - 4a_2}}\right)}{\sqrt{a_1^2 - 4a_2}},\]
which is transcendental.
Otherwise,
\begin{align*}
\int{(1+a_1x+a_2x^2)^q dx} = 
\frac{1}{a_2(q+1)}\,
2^{q-1}
\bigl(2a_2x+a_1-\sqrt{a_1^2-4a_2}\bigr)\,
\left(
   \frac{2a_2x+a_1+\sqrt{a_1^{2}-4a_2}}{\sqrt{a_1^{2}-4a_2}}
\right)^{-q}&\\
\bigl(a_2x^{2} + a_1x + 1\bigr)^q
{}_2F_{1}\!\left([
      -q,\;
     q+1],\;
      [q+2];\;
      \frac{-a_1-2a_2x+\sqrt{a_1^{2}-4a_2}}{2\sqrt{a_1^{2}-4a_2}}
\right)
\;+\;C.
\end{align*}
In other words, we get a particular $_2F_1$ multiplied by an algebraic function. Therefore, the answer to ``Is $\int{(1+ax+bx^2)^{q}dx}$ algebraic?'' depends solely on this $_2F_1$.

\cref{cor:gaussalg} allows to establish all cases when ${}_2F_{1}\!\left([
      -q,\;
     q+1],\;
      [q+2];\;
      \frac{-a_1-2a_2x+\sqrt{a_1^{2}-4a_2}}{2\sqrt{a_1^{2}-4a_2}}
\right)$  is algebraic:
\begin{enumerate}
\item $q \in \mathbb{N}$, or
\item $q + \frac{3}{2} \in -\mathbb{N}$.
\end{enumerate}

\item[1.2] Assume $\mathbf{a_1^2-4a_2<0}$.

If $q=-1,$
\[\int{(1+a_1x+a_2x^2)^{-1} dx} = \frac{2\operatorname{artanh}\left
(\frac{2a_2x + a_1}{\sqrt{-a_1^2 + 4a_2}}\right)}{\sqrt{-a_1^2 + 4a_2}}.\]
which is transcendental.
Otherwise, the integral admits a solution in terms of algebraic functions and a Gaussian $_2F_1$:
\begin{align*}
\int{(1+a_1x+a_2x^2)^q dx} = 
\left(x+\frac{a_1}{2a_2}\right)\left(\frac{4a_2-a_1^2}{4a_2}\right)^q\,
{}_2F_{1}\!\left(\left[
      \frac{1}{2},\;
     -q\right],\;
      \left[\frac{3}{2}\right];\;
      \frac{(a_1+4a_2x)^2}{a_1^2-4a_2}
\right)
\;+\;C.
\end{align*}
which is algebraic if and only if:
\begin{enumerate}
\item $q \in \mathbb{N}$, or
\item $q + \frac{3}{2} \in -\mathbb{N}$.
\end{enumerate}
\end{itemize}
Observe that algebraicity conditions turn out to be the same for either of the subcases.

\item $\mathbf{n \geq 1, q\notin -\mathbb{N}^*, q+\frac{3}{2}\notin -\mathbb{N}}$.
\begin{lem}\label{lem:comb1}
For all $n \in \mathbb{N}, q\in\mathbb{Q}\setminus\{-1,-\frac{3}{2},-2,\ldots,-\frac{n+1}{2}\}$ there exist $c \in \mathbb{Q}$ and $ C(x) \in \mathbb{Q}[x]_{n-1}$ such that
\begin{equation}\label{eq:{4.4}}\tag{4.4}
\int{x^n (1+a_1x+a_2x^2)^q dx} = c\int{(1+a_1x+a_2x^2)^qdx} + C(x)(1+a_1x+a_2x^2)^{q+1}.
\end{equation}
\end{lem}
\begin{proof}
Differentiating both sides of \eqref{eq:{4.4}} yields
\[x^n (1+a_1x+a_2x^2)^q = c(1+a_1x+a_2x^2)^q+ C^\prime(x)(1+a_1x+a_2x^2)^{1+q} + C(x)(1+q)(a_1+2a_2x)(1+a_1x+a_2x^2)^{q}.\]
Divide both sides of the equation above by $(1+a_1x+a_2x^2)^q$:
\begin{equation} \label{eq:{4.5}}\tag{4.5}
x^n = c + C^\prime(x)(1+a_1x+a_2x^2) + C(x)(a_1+2a_2x)(1+q).
\end{equation}

We write $C(x) = \sum\limits_{i=0}^{n-1}c_ix^i$ and show that it is always possible to find $c, c_0, c_1, \ldots, c_{n-1} \in \mathbb{Q}$ such that equation \eqref{eq:{4.5}} holds. 
Each side of \eqref{eq:{4.5}} is simply a polynomial of degree $n$. \\
If $n=1$, then one solves
\[x = c + c_0(a_1+2a_2x)(1+q).\]
for $c$ and $c_0$ and gets
\[c_0 = \frac{1}{2a_2(1+q)},\]
\[c = -\frac{a_1}{2a_2}.\]

Otherwise, write
\[x^n = c + (1+a_1x+a_2x^2)\sum\limits_{i=0}^{n-2}(i+1)c_{i+1}x^{i} + (1+q)(a_1+2a_2x)\sum\limits_{i=0}^{n-1}c_ix^i.\]
Rearranging the terms gives:
\begin{align*}
x^n &= c + c_1 + a_1c_0(q+1)  \\
&+ \sum\limits_{i=1}^{n-2} (c_{i+1}(i+1) + c_i \left(a_1i + a_1(q+1)\right) + c_{i-1} \left(a_2(i-1) + 2a_2(q+1))\right)x^i \label{4.6}\tag{4.6}\\
&+ (c_{n-1}\left(a_1(n-1)+a_1(q+1)\right)+c_{n-2}\left(a_2(n-2) + 2a_2(q+1)\right))x^{n-1} \\
& + (a_2(n-1)c_{n-1}+2(q+1)a_2c_{n-1})x^n.
\end{align*}

One can now derive explicit expressions for $c,c_0,\ldots,c_{n-1}$:
\begin{align*}
c_{n-1} &= \frac{1}{a_2(n+1+2q)}, \\
c_{n-2} &= -\frac{a_1(n+q)}{a_2(n+2q)}c_{n-1}, \\
c_{i} &= -\frac{c_{i+2}(i+2) + c_{i+1}a_1\left(i + 2 + k\right)}{a_2(i + 2 + 2q)}, \quad i=0,\ldots,n-3 \\
c &= -c_1 - a_1c_0(q+1).
\end{align*}

Starting from $c_{n-1}$, one can iteratively compute all $c_i$ for $i=0,\ldots,n-1$, as well as $c$. This procedure only fails if a zero appears in a denominator during computation. Since $a_2 \neq 0$ by assumption, we need to check that $i + 2 + 2q \neq 0$ for $i=0,\ldots,n-1$. Equivalently, that
\[q \neq -1 - \frac{i}{2} \text{ for } i = 0,\ldots,n-1.\]
\\
This is precisely where the condition $q\in\mathbb{Q}\setminus\{-1,-\frac{3}{2},-2,\ldots,-\frac{n+1}{2}\}$ comes from: it ensures that $c, c_0, \ldots, c_{n-1}$ are well defined.
\end{proof}
Hence, for $q \in \mathbb{Z}\setminus\{-1,-\frac{3}{2},\ldots,-\frac{n+1}{2}\}$ $\int{x^n (1+a_1x+a_2x^2)^q dx}$ is algebraic if and only if $\int{(1+ax+bx^2)^qdx}$ is algebraic, or $c = 0$. Using our classification for the case $n=0$, we deduce the following: 
\begin{cor}\label{cor6.1}
If $n \in \mathbb{N}, q\notin -\mathbb{N}^*, q+\frac{3}{2}\notin -\mathbb{N}$, then $\int{x^n (1+a_1x+a_2x^2)^q dx}$ is algebraic if and only if 
\begin{enumerate}
\item $q\in\mathbb{N}$, or
\item $c = 0$, where $c$ is the constant from \eqref{eq:{4.4}}.
\end{enumerate}
\end{cor}

\begin{proof}
By \eqref{eq:{4.4}}, 
\[\int{x^n (1+a_1x+a_2x^2)^qdx} = c\int{(1+a_1x+a_2x^2)^qdx} + \underbrace{C(x)(1+a_1x+a_2x^2)^{q+1}}_{\text{algebraic}}.\]

The corollary follows immediately from the fact that the set of algebraic functions is closed under addition.
\end{proof}

\item $\mathbf{n \geq 1, q+\frac{3}{2}\in -\mathbb{N}}$.
There are three possible scenarios depending on $n$: 
\begin{itemize}
\item[3.1] $\mathbf{n < -2q-1}$, i.e. $q < -\frac{n+1}{2}$. Then we simply apply \cref{cor6.1} to deduce that $\int{x^n (1+ax+bx^2)^qdx}$  is algebraic if and only if $\int{ (1+ax+bx^2)^qdx}$ is algebraic. Note that $\int{ (1+ax+bx^2)^qdx}$ is always algebraic by the case analysis performed in 2.
\item[3.2] $\mathbf{n = -2q-1}$. Equivalently, $q = -\frac{n+1}{2}$. Since $q+\frac{3}{2}\in -\mathbb{N}$, $n$ must be even.

\begin{lem}\label{lem:4.4}
\[\int x^n(1+a_1x+a_2x^2)^{-\frac{n+1}{2}}dx, \quad a_1,a_2 \in \mathbb{Q}^*, \ n \in 2\mathbb{N} \]
is not globally bounded.
\end{lem}
Our proof technique mirrors the computation from \cite[p.~71]{KL05}.
\begin{proof}
Let $(1+a_1x+a_2x^2)^{-\frac{n+1}{2}} = \sum\limits_{k\geq0}d_kx^k$. Then
\[\int x^n(1+a_1x+a_2x^2)^{-\frac{n+1}{2}}dx = \int \sum\limits_{k\geq0}d_kx^{n+k} dx = \sum\limits_{k\geq n+1}\frac{d_{k-n-1}}{k}x^{k}.\]

\begin{align*}
(1+a_1x+a_2x^2)^{-\frac{n+1}{2}} &= \sum\limits_{k\geq0}\binom{-(n+1)/2}{k}(a_1x+a_2x^2)^k \\&= \sum\limits_{k\geq0}\binom{-(n+1)/2}{k}x^k\sum\limits_{i=0}^{k}\binom{k}{i}x^ia_1^{k-i}a_2^i, \\
d_m = [x^m] \quad (1+a_1x+a_2x^2)^{-\frac{n+1}{2}} &= \sum\limits_{k=0}^{m}\binom{-(n+1)/2}{k} \binom{k}{m-k}a_1^{2k-m}a_2^{m-k}.
\end{align*}
In addition,
\begin{align*}
\binom{-(n+1)/2}{k} &= \frac{-\frac{n+1}{2}(-\frac{n+1}{2}-1)\ldots(-\frac{n+1}{2}-k+1)}{k!} \\ &= \frac{(-n-1)(-n-3)\ldots(-n-2k+1)}{2^kk!} \\ &= \frac{(-1)^{k}}{2^k}\frac{(n+1)(n+3)\ldots(n+2k-1)}{k!}.
\end{align*}
Let $p > 3$ be a prime number, pairwise coprime with $a_1$ and $a_2$. Then, for $k$ between 0 and $p$:
\begin{align*}
\binom{-(n+1)/2}{k} \not\equiv 0 \text{ mod } p \iff n+2k-1 < p &\iff 2k < p - n + 1,\\
\binom{k}{p-n-1-k} = 0 \text{ for } k < p-n-1-k &\iff 2k < p-n-1.
\end{align*}

Consequently, \[\binom{-(n+1)/2}{k}\binom{n}{p-n-1-k} \not\equiv 0 \text{ mod } p \] if and only if 
\[p-n-1 \leq 2k < p-n+1.\]
Observe that since $p$ is odd and $n$ is even, the condition above is equivalent to $k = \frac{p-n-1}{2}$.

Hence, if $m = p-n-1$, the only product of the form 
\[\binom{-(n+1)/2}{k} \binom{k}{p-n-1-k}\]
that appears in $d_{p-n-1}$ and is not divisible by $p$ is the one corresponding to $k = \frac{p-n-1}{2}$. Therefore, $d_{p-n-1}$ is not divisible by $p$, which proves that the denominators of $\frac{d_{k-n-1}}{k}$ become divisible by arbitrarily large primes. So $\sum\limits_{k\geq n+1}\frac{d_{k-n-1}}{k}x^{k}$ is not almost integral.
\end{proof}
Hence, $\int{x^{-2q-1} (1+ax+bx^2)^qdx}$ is not globally bounded and not algebraic.
\item[3.3] $\mathbf{n > -2q-1}$. Essentially, this case is not very different from the one where $q\in\mathbb{Q}\setminus\{-1,-\frac{3}{2},-2,\ldots,-\frac{n+1}{2}\}$, as the following identity holds:
\begin{align*}\label{eq:{4.7}}\tag{4.7}
\int{x^n (1+a_1x+a_2x^2)^q dx} &= \underbrace{c\int{(1+a_1x+a_2x^2)^qdx}}_{\text{algebraic}} \\&+ \tilde{c}\underbrace{\int{x^{-2q-1}(1+a_1x+a_2x^2)^qdx}}_{\text{transcendental}} \\&+ \underbrace{C(x)(1+a_1x+a_2x^2)^{q+1}}_{\text{algebraic}}.
\end{align*}
Where $c, \tilde{c} \in \mathbb{Q}$ and $ C(x) \in \mathbb{Q}[x]_{n-1}$. In particular, this implies $\int{x^n (1+a_1x+a_2x^2)^q dx}$ is not algebraic unless $\tilde{c} = 0$.

The proof is analogous to that of \eqref{eq:{4.4}} except for the computation of $c_{-2q-2}$. Similarly, write
\[x^n = c + \tilde{c}x^{-2q-1}+ (1+a_1x+a_2x^2)\sum\limits_{i=0}^{n-2}(i+1)c_{i+1}x^{i} + (1+q)(a_1+2a_2x)\sum\limits_{i=0}^{n-1}c_ix^i.\]
We apply the same coefficient matching but with additional ``patching'' at $x^{-2q-1}$:
If $n = -2q$, 
\begin{align*}
x^{-2q} &= c + c_1 + a_1c_0(q+1)  \\
&+ \sum\limits_{1 \leq i \leq -2q-2} ((i+1)c_{i+1} + a_1(i+q+1)c_i + a_2(i+2q+1)c_{i-1})x^i \\
& +(\tilde{c}-a_1qc_{-2q-1})x^{-2q-1} \\
& + a_2c_{-2q-1}x^{-2q}.
\end{align*}

Otherwise,
\begin{align*}
x^n &= c + c_1 + a_1c_0(q+1)  \\
&+ \sum\limits_{\substack{1 \leq i \leq n-2 \\ i \neq -2q-1}} ((i+1)c_{i+1} + a_1(i+q+1)c_i + a_2(i+2q+1)c_{i-1})x^i \\
&+ (\tilde{c}-2qc_{-2q} -a_1qc_{-2q-1})x^{-2q-1} \\
&+ (a_1(n+q)c_{n-1}+a_2(n+2q)c_{n-2})x^{n-1} \\
& + (a_2(n+2q+1)c_{n-1})x^n.
\end{align*}
Observe that previously $c_{-2q-2}$ was computed from the equation \[(i+1)c_{i+1} + a_1(i+q+1)c_i + a_2(i+2q+1)c_{i-1}= 0, \quad i = -2q-1.\] 
Or, in the case $n = -2q$, from
\[a_1(n+q)c_{n-1}+a_2(n+2q)c_{n-2}= 0, \quad n = -2q.\]  
However, when $i = -2q-2$, the coefficient of $c_{-2q-2}$ in the expressions above becomes zero. This is when $\tilde{c}$ comes into play: it allows to cancel out the coefficient of $x^{-2q-1}$. Hence, if $n \neq -2q$:
\begin{align*}
c_{n-1} &= \frac{1}{a_2(n+1+2q)}, \\
c_{n-2} &= -\frac{a_1(n+q)}{a_2(n+2q)}c_{n-1}, \\
c_{i} &= -\frac{c_{i+2}(i+2) + c_{i+1}a_1\left(i + 2 + q\right)}{a_2(i + 2 + 2q)}, \quad i=0,\ldots,n-3, \ i \neq -2q-2, \\
c_{-2q-2} &=  0 \quad \text{(could be any value)},\\
c &= -c_1 - a_1c_0(q+1), \\
\tilde{c} &=  2qc_{-2q} + a_1qc_{-2q-1}.
\end{align*}
If $n = -2q$:
\begin{align*}
c_{-2q-1} &= \frac{1}{a_2}, \\ 
c_{-2q-2} &=  0 \quad \text{(could be any value)},\\
c_{i} &= -\frac{c_{i+2}(i+2) + c_{i+1}a_1\left(i + 2 + q\right)}{a_2(i + 2 + 2q)}, \quad i=0,\ldots,-2q-3, \\
c &= -c_1 - a_1c_0(q+1), \\
\tilde{c} &=  a_1qc_{-2q-1}.
\end{align*}
\end{itemize}

\item $\mathbf{n\ge1, q\in -\mathbb{N}^*}$.
Since the double root case was treated earlier, we may assume 
\[1+a_1x+a_2x^2 = (1+\alpha x)(1+\beta x), \ \alpha \neq \beta.\] Moreover, $\alpha, \beta \neq 0$. Apply a partial fraction decomposition to the integrand:
\begin{equation}\label{eq:{4.8}} \tag{4.8}
\frac{x^n}{(1+a_1x+a_2x^2)^{-q}} = \frac{x^n}{(1+\alpha x)^{-q} (1+\beta x)^{-q}} = \sum\limits_{m = 1}^{-q} \left(\frac{A_m}{(1+\alpha x)^m} + \frac{B_m}{(1+\beta x)^m} \right).
\end{equation}
Where $A_m$ and $B_m$ are coefficients to be determined. Let us integrate the sum in \eqref{eq:{4.8}}:
\begin{align*} \label{eq:{4.9}} \tag{4.9}
&\int{\frac{x^n}{(1+a_1x+a_2x^2)^{-q}} dx} = \int{\sum\limits_{m = 1}^{-q} \left(\frac{A_m}{(1+\alpha x)^m} + \frac{B_m}{(1+\beta x)^m}\right) dx} \\&= A_1 \frac{\ln(1+\alpha x)}{\alpha} + B_1 \frac{\ln(1+\beta x)}{\beta} + \sum\limits_{m = 2}^{-q} \left(\frac{A_m}{\alpha(1-m)(1+\alpha x)^{m-1}} + \frac{B_m}{\beta(1-m)(1+\beta x)^{m-1}}\right).
\end{align*}
In other words, the result consists of a sum of rational functions and logarithmic terms. The only scenario in which \eqref{eq:{4.9}} can be algebraic is if $A_1 = B_1 = 0$. The coefficients $A_1$ and $B_1$ can be computed either via the generalized residue formula:
\[
A_1 = \frac{1}{(-q-1)!}
  \left.\frac{d^{\,-q-1}}{dx^{\,-q-1}}
    \Bigl[\frac{x^n}{(1+\beta x)^{-q}}\Bigr]
  \right|_{x=-\frac{1}{\alpha}} \neq 0,\]
\[
B_1 = \frac{1}{(-q-1)!}
  \left.\frac{d^{\,-q-1}}{dx^{\,-q-1}}
    \Bigl[\frac{x^n}{(1+\alpha x)^{-q}}\Bigr]
  \right|_{x=-\frac{1}{\beta}}
\neq 0.\]
or symbolically, as they naturally arise when integrating rational functions, for instance using Hermite reduction \cite[Chapter 22]{gerhard2013modern}. This is a standard and largely automated procedure in computer algebra systems.
\end{enumerate}

\begin{rem}
Although the case $b_0=0$ falls outside the class we are considering, it is still easy to treat. $S(x)$ writes as
\[S(x) = (a_2x^2 + a_1x + 1)^{-1 - \frac{b_1}{a_1}}\left(((a_1+b_1)s_0+s_1)\int(a_2x^2 + a_1x + 1)^{\frac{b_1}{a_1}} + c_1\right).\]
Hence, if $\frac{b_1}{a_1} \in \mathbb{N}$ or $\frac{b_1}{a_1}+\frac{3}{2} \in -\mathbb{N}$, then \cref{case:alg} by the analysis performed for $n=0$.
Otherwise, \cref{case:line} with $(1,-a_1-b_1)$.
\end{rem}

\subsubsection{Transcendence in both components.}
\begin{flushright}
\footnotesize\itshape
``Though this be madness, yet there is method in’t.''\\
\textup{--- W. Shakespeare, \textit{Hamlet}, Act II, Scene II.}
\end{flushright}
Even if both $I_1$ and $I_2$ are transcendental, the linear combination
\[
b_0s_0 I_1 + (a_1s_0 + b_0s_1 + b_1s_0 + s_1)I_2
\]
might still be algebraic. In this section, we provide an algorithm to detect such cases.
\begin{enumerate}
\item $\mathbf{q\notin -\mathbb{N}^*, q+\frac{3}{2}\notin -\mathbb{N}}$. Then, by \cref{lem:comb1}:
\begin{align*}
I_1 &= c_{I}\int{(1+a_1x+a_2x^2)^qdx} + C_1(x)(1+a_1x+a_2x^2)^{q+1},\\
I_2 &= c_{II}\int{(1+a_1x+a_2x^2)^qdx} + C_2(x)(1+a_1x+a_2x^2)^{q+1}.
\end{align*}
The transcendence of $I_1$ and $I_2$ means two things: $\int{(1+a_1x+a_2x^2)^qdx}$ is transcendental, and $c_I,c_{II}\neq 0$. So
\begin{align*}
& b_0s_0 I_1 + (a_1s_0 + b_0s_1 + b_1s_0 + s_1)I_2 =\\ &= (b_0s_0c_I + (a_1s_0 + b_0s_1 + b_1s_0 + s_1)c_{II})\underbrace{\int{(1+a_1x+a_2x^2)^qdx}}_{\text{transcendental}} +\\&+ \underbrace{(b_0s_0C_1(x)+(a_1s_0 + b_0s_1 + b_1s_0 + s_1)C_2(x))(1+a_1x+a_2x^2)^{q+1}}_{\text{algebraic}}.
\end{align*}
can only be algebraic if 
\[b_0s_0c_I + (a_1s_0 + b_0s_1 + b_1s_0 + s_1)c_{II} = 0,\] 
which rewrites as 
\[s_1=\frac{s_0(b_0c_I + a_1c_{II}+b_1c_{II})}{(b_0+1)c_{II}}.\]
Hence, we are in \cref{case:line} and the algorithm returns $(c_{II}(b_0+1),b_0c_I + a_1c_{II}+b_1c_{II})$.
\item $\mathbf{q+\frac{3}{2}\in -\mathbb{N}}$.

Each of \( I_1 \), \( I_2 \) can be written in the form
\[
c\underbrace{\int{(1+a_1x+a_2x^2)^qdx}}_{\text{algebraic}} + \tilde{c}\underbrace{\int{x^{-2q-1}(1+a_1x+a_2x^2)^qdx}}_{\text{transcendental}} + \underbrace{C(x)(1+a_1x+a_2x^2)^{q+1}}_{\text{algebraic}},
\]
where the constant $\tilde{c}$ will be denoted as \( \tilde{c}_1 \) (respectively, \( \tilde{c}_2 \)) for \( I_1 \)(respectively, \( I_2 \)). Hence,
\begin{align*}
& b_0s_0 I_1 + (a_1s_0 + b_0s_1 + b_1s_0 + s_1)I_2 \\
&= (b_0s_0\tilde{c}_1 + (a_1s_0 + b_0s_1 + b_1s_0 + s_1)\tilde{c}_2)\int{x^{-2q-1}(1+a_1x+a_2x^2)^qdx} + \mathcal{A}(x).
\end{align*}
where \( \mathcal{A}(x) \) is some algebraic function. Similarly to the previous case, we are in \cref{case:line} and the algorithm returns $(\tilde{c}_2(b_0+1),b_0\tilde{c}_1 + a_1\tilde{c}_2+b_1\tilde{c}_2)$.
\item $\mathbf{q\in -\mathbb{N}^*}$. Then
\begin{align*}
&I_1 = A_1^{(1)} \frac{\ln(1+\alpha x)}{\alpha} + B_1^{(1)} \frac{\ln(1+\beta x)}{\beta} + \mathcal{A}_1(x), \\&
I_2 = A_1^{(2)} \frac{\ln(1+\alpha x)}{\alpha} + B_1^{(2)} \frac{\ln(1+\beta x)}{\beta} + \mathcal{A}_2(x).
\end{align*}
where \( \mathcal{A}_1(x),\mathcal{A}_2(x)  \) are some algebraic functions. So
\begin{align*}
b_0s_0 I_1 + (a_1s_0 + b_0s_1 + b_1s_0 + s_1)I_2 &= 
(b_0s_0A_1^{(1)} + (a_1s_0 + b_0s_1 + b_1s_0 + s_1)A_1^{(2)})\frac{\ln(1+\alpha x)}{\alpha} \\
&+ (b_0s_0B_1^{(1)} + (a_1s_0 + b_0s_1 + b_1s_0 + s_1)B_1^{(2)})\frac{\ln(1+\beta x)}{\beta} \\
&+  \mathcal{A}(x).
\end{align*}
Equivalently,
\begin{align*}
b_0s_0 I_1 + (a_1s_0 + b_0s_1 + b_1s_0 + s_1)I_2 &= 
(s_0(b_0A_1^{(1)} + a_1A_1^{(2)} + b_1A_1^{(2)}) + s_1A_1^{(2)}(b_0+1))\frac{\ln(1+\alpha x)}{\alpha} \\
&+ (s_0(b_0B_1^{(1)} + a_1B_1^{(2)} + b_1B_1^{(2)}) + s_1B_1^{(2)}(b_0+1))\frac{\ln(1+\beta x)}{\beta} \\
&+  \mathcal{A}(x).
\end{align*}
which is algebraic if and only if the coefficients multiplying the logarithmic terms are simultaneously equal to zero (we recall that $\alpha \neq \beta$). In other words, one needs 
\[
\begin{pmatrix}
b_0A_1^{(1)} + (a_1+b_1)A_1^{(2)} &  (b_0+1)A_1^{(2)}\\
b_0B_1^{(1)} + (a_1+b_1)B_1^{(2)} & (b_0+1)B_1^{(2)}
\end{pmatrix} 
\begin{pmatrix}
s_0\\s_1
\end{pmatrix} = 
\begin{pmatrix}
0\\0
\end{pmatrix}.
\]
Hence, we are in \cref{case:line} if the matrix
\[\mathcal{M} = 
\begin{pmatrix} 
b_0A_1^{(1)} + (a_1+b_1)A_1^{(2)} &  (b_0+1)A_1^{(2)}\\
b_0B_1^{(1)} + (a_1+b_1)B_1^{(2)} & (b_0+1)B_1^{(2)}
\end{pmatrix}
\]
is not invertible. Then the algorithm returns a nonzero element of its kernel. Otherwise, the conclusion is \cref{case:trans}.
\end{enumerate}

\begin{figure}[H]
\centering
\includegraphics[scale=0.45]{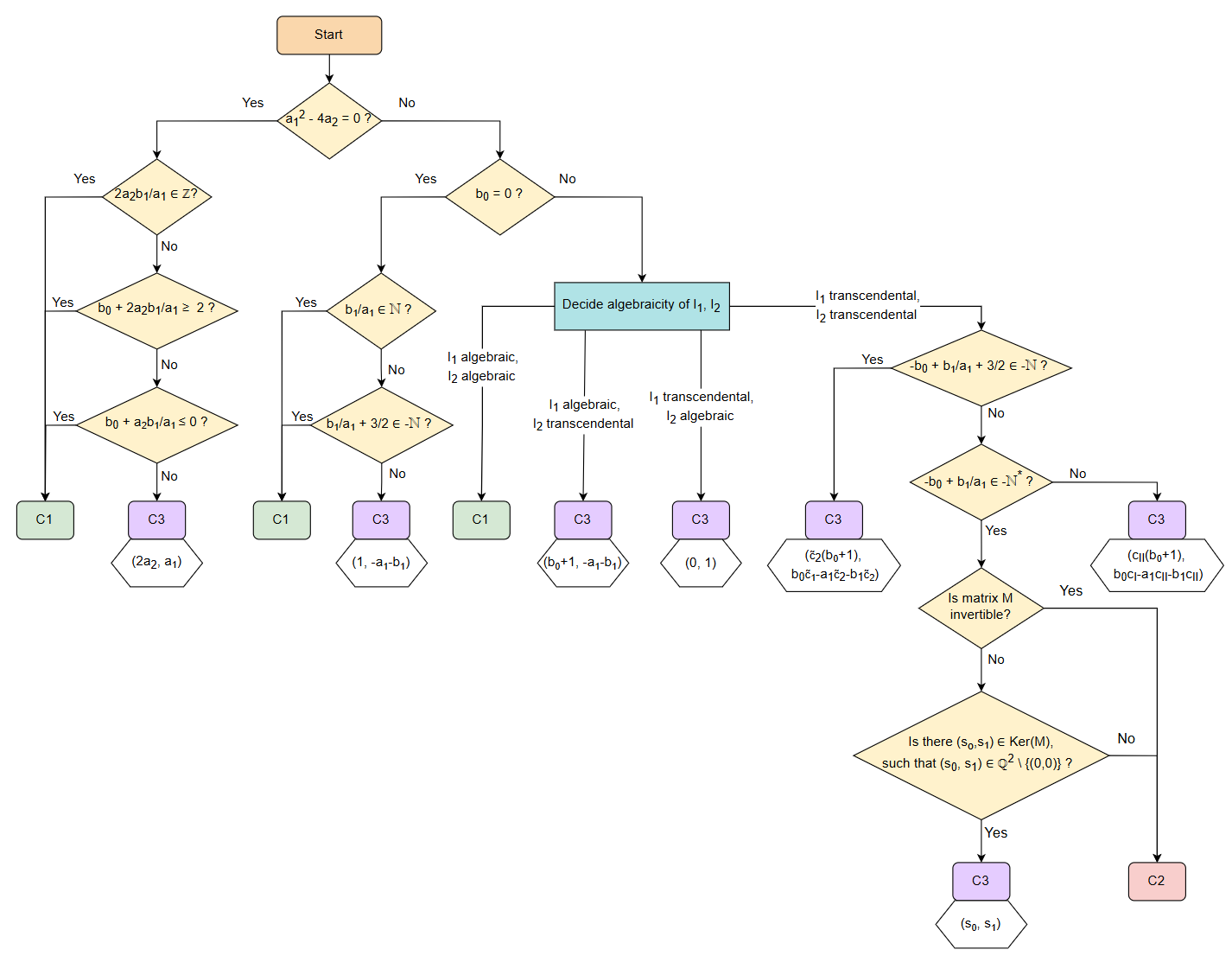}
\caption{Deciding which of the cases \cref{case:alg}, \cref{case:trans}, \cref{case:line} holds for a recurrence relation of type \eqref{eq:{4.1}} with $b_2 = \frac{2a_2b_1-a_1a_2b_0}{a_1}$. If the answer is \cref{case:line}, a pair $(s_0,s_1) \in \mathbb{Q}^2\setminus\{(0,0)\}$ such that $S(x)$ is algebraic is provided. The step ``Decide algebraicity of $I_1,I_2$'' is described in \cref{fig:alg}}.
\label{fig:main algo}
\end{figure}

\subsubsection{Examples}
\begin{enumerate}
\item \textbf{Motzkin numbers} (\href{https://oeis.org/A001006}{OEIS 
A001006}).

The recurrence
\[(n+2)m_n = (2n+1)m_{n-1}+(3n-3)m_{n-2}\]
corresponds to
\[b_0=2,a_1=-2,b_1=-1,a_2=-3,b_2=3.\]
We have $b_2 = \frac{2a_2b_1-a_1a_2b_0}{a_1}$, $a_1^2-4a_2\neq0$, $b_0\in\mathbb{N}^*$. Moreover, $q=-b_0+\frac{b_1}{a_1} = -\frac{3}{2} \implies q+\frac{3}{2}=0\in-\mathbb{N}$. 
\[I_1 = \int{x(1-2x-3x^2)^{-\frac{3}{2}}dx}, \qquad I_2 = \int{x^2(1-2x-3x^2)^{-\frac{3}{2}}dx}.\]
$b_0-1=1 < -2q-1 = 2 \implies$ $I_1$ is algebraic.\\
$b_0=2 = -2q-1 \implies$ $I_2$ is transcendental.

The algorithm returns \cref{case:line} with $(b_0+1,-(a_1+b_1)) = (3,3)$, proving that Motzkin-type sequences are globally bounded if and only if the two first terms coincide.

Alternatively, one can verify the transcendence of any solution with $m_0\neq m_1$ by initializing the recurrence with, for instance, $m_0=0,m_1=1$, converting it into a homogeneous differential equation, and then using the \textit{istranscendental} function from Maple’s \href{https://perso.ens-lyon.fr/bruno.salvy/software/the-gfun-package/}{gfun} package (available since version 4.10).
\item \textbf{Central trinomial coefficients } (\href{https://oeis.org/A002426}{OEIS A002426}).

The recurrence
\[ns_n = (2n-1)s_{n-1} + (3n-3)s_{n-2}\]
corresponds to
\[b_0=0,a_1=-2,b_1=1,a_2=-3,b_2=3.\]
We have $b_2 = \frac{2a_2b_1-a_1a_2b_0}{a_1}$, $a_1^2-4a_2\neq0$. $b_0=0$, so we look at $\frac{b_1}{a_1}=-\frac{1}{2}$. Since $\frac{b_1}{a_1} = \frac{1}{2} \notin \mathbb{N}$ and $\frac{b_1}{a_1}+\frac{3}{2} = 1 \notin -\mathbb{N}$, the algorithm returns \cref{case:line} with $(1,-(a_1+b_1)) = (1,1)$.

This is consistent with the fact that $s_0 = s_1 = 1$ yields not just an almost integral, but an integer sequence of the Central trinomial coefficients. In particular, we deduce that the only integer sequences satisfying the recurrence for the central trinomial coefficients are those with $s_0=s_1$, mirroring the case of Motzkin numbers.
\item \textbf{Large Schröder numbers} (\href{https://oeis.org/A006318}{OEIS A006318}) \textbf{: fully algebraic example.}

The recurrence
\[(n+1)s_n = (6n-3)s_{n-1}-(n-2)s_{n-2}\]
corresponds to
\[b_0=1,a_1=-6,b_1=3,a_2=1,b_2=-2.\]
We have $b_2 = \frac{2a_2b_1-a_1a_2b_0}{a_1}$, $a_1^2-4a_2\neq0$, $b_0\in\mathbb{N}^*$. Moreover, $q=-b_0+\frac{b_1}{a_1} = -\frac{3}{2} \implies q+\frac{3}{2}\in\mathbb{N}$. 
\[I_1 = \int{(1-6x+x^2)^{-\frac{3}{2}}dx}, \qquad I_2 = \int{x(1-6x+x^2)^{-\frac{3}{2}}dx}.\]
$b_0-1=0 < -2q-1 = 2 \implies$ $I_1$ is algebraic.\\
$b_0=1 = -2q-1 \implies$ $I_2$ is algebraic.

The algorithm returns \cref{case:alg}.
\item \textbf{Fully transcendental example}.

Consider the following innocent-looking recurrence:
\[(n+3)s_n = (n+1)s_{n-1} + (2n-2)s_{n-2} .\]
It corresponds to
\[b_0=3,a_1=-1,b_1=-1,a_2=-2,b_2=2.\]
We have $b_2 = \frac{2a_2b_1-a_1a_2b_0}{a_1}$, $a_1^2-4a_2\neq0$, $b_0\in\mathbb{N}^*$ and $q=-b_0+\frac{b_1}{a_1} = -2 \in-\mathbb{N}^*$. 
\[I_1 = \int{\frac{x^2}{(1-x-2x^2)^2}dx} = \frac{2}{27}\ln(2x-1) -\frac{2}{27}\ln(x+1) - \frac{1}{18(2x-1)}-\frac{1}{9(x+1)},\]
\[I_2 = \int{\frac{x^3}{(1-x-2x^2)^2}dx} = \frac{7}{108}\ln(2x-1)+\frac{5}{27}\ln(x+1) - \frac{1}{36(2x-1)}+\frac{1}{9(x+1)}.\]
So,
\[A_1^{(1)}=\frac{2}{27}, \ B_1^{(1)}=-\frac{2}{27}, \ A_1^{(2)}=\frac{7}{108}, \ B_1^{(2)}=\frac{5}{27}.\]
None of these numbers are equal to zero, in particular this implies that both $I_1$ and $I_2$ are transcendental. Moreover, the matrix
\[
\begin{pmatrix}
b_0A_1^{(1)} + (a_1+b_1)A_1^{(2)} &  (b_0+1)A_1^{(2)}\\
b_0B_1^{(1)} + (a_1+b_1)B_1^{(2)} & (b_0+1)B_1^{(2)}
\end{pmatrix} 
=
\begin{pmatrix}
5/54 &  7/27\\
-16/27 & 20/27
\end{pmatrix}. 
\]
is invertible, as its determinant equals $\frac{2}{9}\neq0$. Hence, no choice of $s_0,s_1$ can make $S(x)=\sum\limits_{n\ge0}s_nx^n$ algebraic. The algorithm returns \cref{case:trans}.
\end{enumerate}

\section{Integrality revisited }\label{sec:integrality}
\begin{flushright}
\footnotesize\itshape
``Coming back to where you started is not the same as never leaving.''\\
\textup{--- T. Pratchett, \textit{A Hat Full of Sky}.}
\end{flushright}
Although proving algebraicity does not directly address the problem of finding integer solutions to a recurrence, it often greatly simplifies the analysis. In the specific setting of second-order recurrences of the form
\[(n+b_0)s_n+(a_1n+b_1)s_{n-1}+(a_2n+b_2)s_{n-2}=0\]
with $b_2 = \frac{2a_2b_1-a_1a_2b_0}{a_1}$, integrality implies global boundedness, which in turn is equivalent to algebraicity. Thus, the problem of deciding whether a solution is integral reduces to detecting algebraicity — something that can be done using the algorithm developed in \cref{sec:alg GB} (see \cref{fig:main algo}). Analysing integrality within this algebraic framework is generally easier. 

For instance, if an algebraic solution has a closed-form expression involving rational powers of polynomials, its integrality can be established using the criteria derived by Pomerat and Straub in \cite{PS24} (Theorem 1.1 and multiple examples applied to specific cases).

\begin{ex} Consider a recurrence relation:
\begin{equation}\label{eq:{5}}\tag{5}
ns_n+(2n+3)s_{n-1}+9(n+3)s_{n-2}=0.
\end{equation}

Applying our algorithm (note that $b_0=0$), we find that $S(x) = \sum\limits_{n \geq 0}s_nx^n$ is globally bounded if and only if $s_1 = -5s_0$. In that case, one has
\[S(x) = s_0(1+2x+9x^2)^{-\frac{5}{2}}.\]
Clearly, all integral solutions are among the globally bounded ones. Thus, the problem reduces to identifying the values of $s_0$ for which $s_0(1+2x+9x^2)^{-\frac{5}{2}}$ lies in $\mathbb{Z}[[x]]$. 
\begin{prop}
If $s_0\in\mathbb{Z}$, then
\[s_0(1+2x+9x^2)^{-\frac{5}{2}}\in\mathbb{Z}[[x]].\]
\end{prop}
\begin{proof}
The claim is a direct consequence of the following result:
\begin{lem}(\cite[Example 1.8]{PS24})
Let $a,b \in \mathbb{Z}, \ \lambda \in \mathbb{Q}$. Let $k$ be the denominator of $\lambda$ brought to the lowest terms. Then $(1+ax+bx^2)^\lambda \in \mathbb{Z}[[x]]$ if and only if
\begin{itemize}
\item $a,b \in k \ rad(k) \mathbb{Z}$, or
\item $k = 2\kappa$ and $a,b \in \kappa \ rad(\kappa) \mathbb{Z}$ as well as $(a,b) \equiv (2,1)$ (mod 4).
\end{itemize}
\end{lem}
Here $rad(k)$ denotes the largest squarefree integer dividing $k$ (for example, $rad(24) = 6$).

In our setting $a = 2, b = 9, \lambda = -\frac{5}{2}, k = 2, \kappa = 1$. Since $2,9 \in \mathbb{Z}$ and $(2,9) \equiv (2,1)$ (mod 4) we get $(1+2x+9x^2)^{-\frac{5}{2}} \in \mathbb{Z}[[x]]$. Hence, for any $s_0 \in \mathbb{Z}$, $s_0(1+2x+9x^2)^{-\frac{5}{2}} \in \mathbb{Z}[[x]]$. 
\end{proof}

Since taking $s_0 \notin \mathbb{Z}$ would already violate the condition $s_n \in \mathbb{Z} \ \text{ for all } n \in \mathbb{N}$, we conclude that $S(x) \in \mathbb{Z}[[x]]$ if and only if $s_1 = -5s_0$ and $s_0 \in \mathbb{Z}$.
\end{ex}

The next figure summarizes the classification of solutions according to their arithmetic properties. Two solutions are called \emph{independent} if one is not a rational multiple of the other, which in our case translates to the condition $s_0\tilde{s}_1 \neq \tilde{s}_0s_1$. Remarkably, all six areas in the figure below are nonempty: such recurrence relations occur in practice.

\begin{figure}[H]
\centering
\includegraphics[scale=0.4]{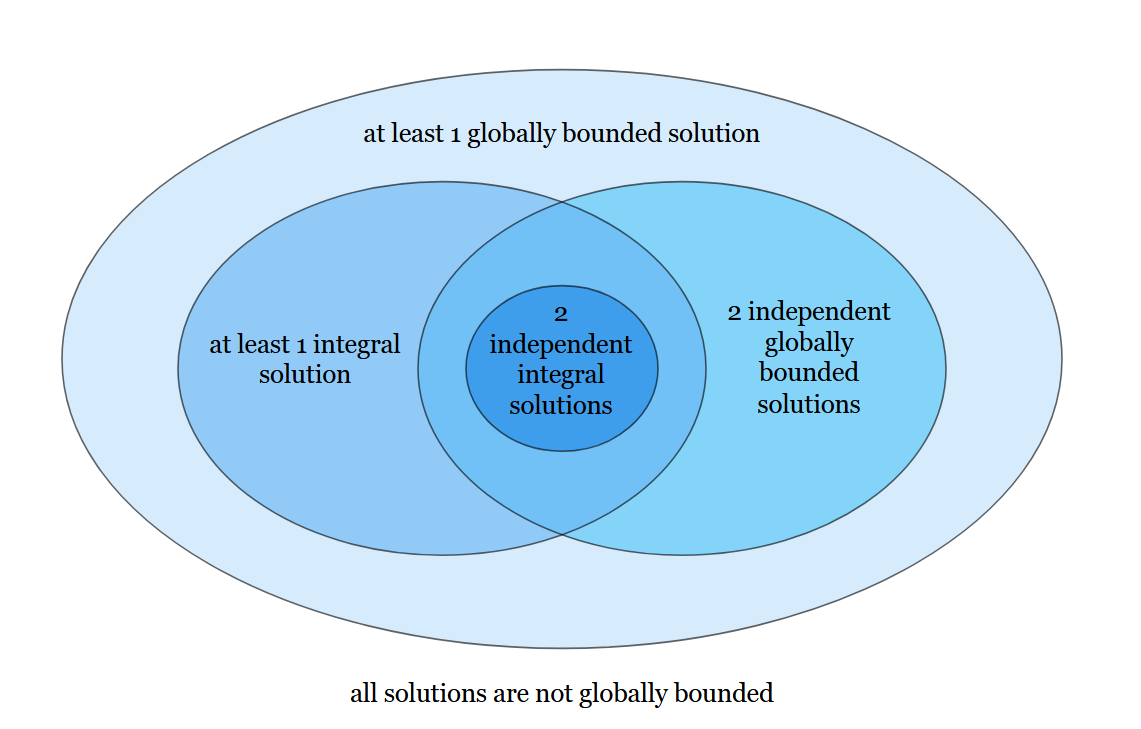}
\caption{Classification of the solution space of second-order linear recurrences with polynomial coefficients, according to integrality and global boundedness properties.}
\end{figure}

\section{Integrality criteria for small Apéry numbers}\label{sec:small Apéry numbers}
In this section, we investigate the sequences $(a_n)_n$ satisfying the recurrence relation 
\begin{equation}\label{eq:{5.1}}\tag{5.1}
n^2a_n = (11n^2-11n+3)a_{n-1} + (n-1)^2a_{n-2}, \quad n \geq 2.
\end{equation}
In the standard setting, this recurrence is initialized with $a_0=1,a_1=3$, yielding the sequence of small Apéry numbers (\href{https://oeis.org/A005258}{OEIS A005258}). These numbers also admit a binomial sum representation:
\begin{equation}\label{eq:{5.2}}\tag{5.2}
a_n = \sum\limits_{k=0}^{n} \binom{n}{k}^2 \binom{n+k}{k},
\end{equation}
which makes it clear that they are all integers. However, the situation becomes more complex when considering arbitrary initial conditions. Let $(a(\mu,\lambda)_n)_n$ denote the solution of \eqref{eq:{5.1}} with $a_0 = \mu, a_1 = \lambda$. The behaviour of these generalized Apéry-type sequences is described in the following theorem, which, to the best of our knowledge, is new, though similar to the case of big Apéry numbers (\href{https://oeis.org/A005259}{OEIS A005259}).

\begin{thm}\label{thm:7.1}
A sequence of rational numbers $(a_n)_n$ satisfying \eqref{eq:{5.1}} is integral if and only if $(a_0,a_1) = (\alpha, 3\alpha)$ for $\alpha \in \mathbb{Z}$.
\end{thm}
\begin{proof}
Let us first show that $a(1,3)_n \in \mathbb{Z}$ for all $n \geq 0$. Obviously, this would imply that $a(\alpha,3\alpha)_n \in \mathbb{Z}$ for all $n \geq 0, \alpha \in \mathbb{Z}$. \\
The sequence $(\tilde{a}_n)_n$ defined by $\tilde{a}_n = \sum\limits_{k=0}^{n} \binom{n}{k}^2 \binom{n+k}{k}$ is a solution to \eqref{eq:{5.1}} (a result due to Apéry). While coming up with such a relation is nontrivial, its verification is straightforward with creative telescoping techniques \cite{zeilberger1991CT}. In Maple this can be done by calling \textit{Zeilberger} function from the \textit{SumTools/Hypergeometric} package, which applies Zeilberger's algorithm to find a recurrence satisfied by a given binomial sum.
\begin{verbatim}
> A := binomial(n, k)^2 * binomial(n + k, k):
> simplify(SumTools[Hypergeometric][Zeilberger](A, n, k, a_n)[1]);
\end{verbatim}
\[(n + 2)^2 a_n^2 + (-11n^2 - 33n - 25)a_n - (n + 1)^2,\]
where the powers of $a_n$ correspond to the coefficient shift, i.e. the expression above is equivalent to \[(n+2)^2a_{n+2} -(11(n+2)^2-11(n+2)+3)a_{n+1}-(n+1)^2a_n=0, \quad n \geq 0\]
which is simply a shift of \eqref{eq:{5.1}}.

In addition, every $\tilde{a}_n$ is a sum of binomials, hence an integer. We check the first two values:
\begin{align*}
\tilde{a}_0 = \binom{0}{0}^2 \binom{0}{0} = 1, \\
\tilde{a}_1 = \binom{1}{0}^2 \binom{1}{0} + \binom{1}{1}^2 \binom{2}{1} = 1+2 = 3.
\end{align*}
Since any solution of \eqref{eq:{5.1}} is uniquely defined by a pair of initial values ($a_0,a_1$), we deduce that $(a(1,3)_n)_n = (\tilde{a}_n)_n$. Thus, $(a(1,3)_n)_n$ is indeed integral.
\end{proof}

We now present two proofs that establish the uniqueness of this integral solution.

\subsection{Proof 1}
For every $n \geq 2$ 
\begin{equation}\label{eq:{5.3}}\tag{5.3}
a_n = \frac{(11n^2-11n+3)}{n^2}a_{n-1} + \frac{(n-1)^2}{n^2}a_{n-2}.
\end{equation}
Consider two sequences $(u_n)_n$ and $(v_n)_n$ defined by
\[u_n = \frac{(11n^2-11n+3)}{n^2}, \quad v_n = \frac{(n-1)^2}{n^2}\]
for all $n \geq 0$. Then \eqref{eq:{5.3}} becomes
\begin{equation}\label{eq:{5.4}}\tag{5.4}
a_n = u_na_{n-1} + v_na_{n-2}.
\end{equation}

The following result by Chang \cite{C84} proves to be remarkably useful in our setting.

\begin{thm}\label{thm:7.2}(\cite{C84}): Suppose that
\begin{enumerate}
    \item[a)]For all integers $n>2$, $v_n \neq 0$.
    \item[b)] The limit $\lim_{n\rightarrow\infty}\prod_{k=2}^{n}|v_k|$ exists.
    \item[c)] The recurrence relation \eqref{eq:{5.4}} has two linearly independent integer solutions.
\end{enumerate}  
Then $|v_n|=1$ for all large n.
\end{thm}

\cref{thm:7.2} is naturally applicable to certain 2-order P-recursive sequences. In particular, it can be used to prove the uniqueness of an integer solution should one exist. The following reasoning mimics the one for big Apéry numbers in \cite{C84}, presented as an application of \cref{thm:7.1} in \cite{C84}. 

Let us prove the $\implies$ direction of \cref{thm:7.1} by contradiction. Assume there exist $\mu,\lambda \in \mathbb{Z}$ with $\lambda \neq 3\mu$ such that $a(\mu,\lambda)_n \in \mathbb{Z}$ for all $n \in \mathbb{N}$. Then all the conditions of \cref{thm:7.2} are satisfied. Indeed, for all $k>2$ 
\[v_k = \frac{(k-1)^2}{k^2} \neq 0\]
and
\[\lim_{n\rightarrow\infty}\prod_{k=2}^{n}|v_k| = \lim_{n\rightarrow\infty}\prod_{k=2}^{n}\frac{(k-1)^2}{k^2} = \lim_{n\rightarrow\infty}\frac{1}{n^2} = 0.\]
By \cref{thm:7.2}, $|v_n|$ must be equal to 1 for all $n > N$ for some $N \in \mathbb{N}$. This is clearly not the case since $v_n = \frac{(n-1)^2}{n^2}$, hence a contradiction. Therefore, \eqref{eq:{5.1}} has at most 1 integral solution, and we have already seen that it is given by $(a(1,3)_n)_n$.
\subsection{Proof 2}
This proof is largely inspired by the work of André-Jeannin \cite{AJ91} on the integrality of big Apéry numbers.

Again, we argue by contradiction. Assume there exist $\mu,\lambda \in \mathbb{Z}$ with $\lambda \neq 3\mu$ such that $a(\mu,\lambda)_n \in \mathbb{Z}$ for all $n \in \mathbb{N}$. By the linearity of \eqref{eq:{5.1}} this implies the integrality of the sequence 
\[(a(0,\lambda-3\mu)_n)_n = (a(\mu,\lambda)_n)_n - \mu(a(1,3)_n)_n.\]
In other words, there exists an integer $c \neq 0$ such that for all$n \in \mathbb{N}\quad a(0,c)_n \in \mathbb{Z}$. To avoid notational ambiguity, we introduce a new sequence $(b_n)_n = (a(0,c)_n)_n$, and write $(a_n)_n = (a(1,3)_n)_n$ for the sequence of small Apéry numbers.\\
Recurrence \eqref{eq:{5.3}} for $(b_n)_n$ can be rewritten as
\begin{equation}\label{eq:{5.5}}\tag{5.5}
b_n - b_{n-1} = \frac{10n^2-11n+3}{n^2}b_{n-1} + \frac{(n-1)^2}{n^2}b_{n-2} = \frac{10(n-\frac{3}{5})(n-\frac{1}{2})}{n^2}b_{n-1} + \frac{(n-1)^2}{n^2}b_{n-2}.
\end{equation}
Now, 
\[\frac{10(n-\frac{3}{5})(n-\frac{1}{2})}{n^2} > 0, \frac{(n-1)^2}{n^2}>0 \ \text{ for } n \geq 2 \text{ and } b_1>b_0=0,\]
so that the RHS of \eqref{eq:{5.5}} is strictly positive for all $n \geq 2$. Consequently, the sequence $(b_n)_n$ is increasing, and so is $(a_n)_n$ by the exact same reasoning. Moreover, $a_n > 0$ and $b_n > 0$ for $n \geq 1$.

We compute $b_2 = \frac{25}{4}b_1$. Therefore, $b_1 = \frac{b_1}{a_1} < \frac{b_2}{a_2} = \frac{25}{12}b_1$. The next step is to prove
\begin{equation}\label{eq:{5.6}}\tag{5.6}
\frac{b_{n-1}}{a_{n-1}} < \frac{b_n}{a_n}.
\end{equation}
for all $n>2$.

Equivalently to \eqref{eq:{5.1}},
\begin{equation}\label{eq:{5.7}}\tag{5.7}
(11n^2 - 11n + 3)a_{n-1} = n^2 a_n - (n-1)^2 a_{n-2}.
\end{equation}

Let $\lambda_i = \frac{b_i}{a_i}$, $i \geq 1$. If $\lambda_{n-2} < \lambda_n$, then from \eqref{eq:{5.7}} we have
\[
(11n^2 - 11n + 3) \lambda_{n-1} a_{n-1} = n^2 \lambda_n a_n - (n-1)^2 \lambda_{n-2} a_{n-2}
< \lambda_n (n^2 a_n - (n-1)^2 a_{n-2}),
\]
and
\[
(11n^2 - 11n + 3) \lambda_{n-1} a_{n-1} > \lambda_{n-2} (n^3 a_n - (n-1)^3 a_{n-2}).
\]
Hence,
\[
\lambda_{n-2} < \lambda_n \implies \lambda_{n-2} < \lambda_{n-1} < \lambda_n.
\]
Similarly,
\[
\lambda_{n-2} \geq \lambda_n \implies \lambda_{n-2} \geq \lambda_{n-1} \geq \lambda_n.
\]

Therefore, $\lambda_{n-2} < \lambda_{n-1} \implies \lambda_{n-1} < \lambda_n$.  
Since \eqref{eq:{5.6}} holds for $n = 2$, it also holds for all $n \geq 2$. 

Substituting $n+1$ in \eqref{eq:{5.7}} and dividing by the same equation written for $(b_n)_n$ yields
\[
\frac{b_n}{a_n} = \frac{(n+1)^2 b_{n+1} - n^2 b_{n-1}}{(n+1)^2 a_{n+1} - n^2 a_{n-1}}.
\]

which up to rearrangement of terms is equivalent to
\begin{equation}\label{eq:{5.8}}\tag{5.8}
(2n + 1) \left( \frac{b_{n+1}}{a_{n+1}} - \frac{b_n}{a_n} \right) a_n a_{n+1}
= n^2 ((a_n b_{n-1} - b_n a_{n-1}) -  (a_n b_{n+1} - b_n a_{n+1})).
\end{equation}

The LHS of \eqref{eq:{5.8}} is positive by \eqref{eq:{5.6}}. Hence, $a_n b_{n+1} - b_n a_{n+1} < a_n b_{n-1} - b_n a_{n-1}$ for all $n \geq 2$.
We deduce from \eqref{eq:{5.6}} that $a_n b_{n-1} - b_n a_{n-1} < 0$ for all $n \geq 2$. Thus,
\begin{equation}\label{eq:{5.9}}\tag{5.9}
a_n b_{n+1} - b_n a_{n+1} < 0 
\end{equation}
for all $n \geq 2$.
But \eqref{eq:{5.6}} directly implies
\begin{equation}\label{eq:{5.10}}\tag{5.10}
a_n b_{n+1} - b_n a_{n+1} > 0
\end{equation}
for all $n \geq 2$. 
The proof follows by observing that \eqref{eq:{5.9}} and \eqref{eq:{5.10}} are incompatible.

\section{Concluding remarks} 
The story, however, is far from complete, as several questions remain open. The case of $b_2 \neq \frac{2a_2b_1-a_1a_2b_0}{a_1}$ is still poorly understood, although we believe it leads to not globally bounded solutions. Moreover, our algorithm is currently not applicable to the recurrences with $b_0 \notin \mathbb{N}$. Beyond that, higher-order recurrences or those with nonlinear polynomial coefficients present further challenges. Their solution spaces are more complex, and the connection between global boundedness and algebraicity becomes harder to characterise.

Although our main goal was to decide algebraicity and global boundedness, the original question of integrality motivated much of this work. By revealing a structural connection between these notions for second-order recurrences, we were able to approach parts of the integrality problem via concepts from differential algebra. This perspective not only simplifies the analysis, but also suggests a path toward a more systematic, algorithmic treatment of arithmetic properties in holonomic sequences.

\bibliographystyle{alpha}
\bibliography{references}
\end{document}